\documentclass{conm-p-l}

\theoremstyle{definition}

\theoremstyle{remark}

\numberwithin{equation}{section}

\usepackage{amsmath,amsfonts, latexsym,graphicx, footmisc, caption, color, tikz-cd}

\newcommand{\U}{{\mathcal U}}
\newcommand{\0}{{\mathbf 0}}
\newcommand{\C}{{\mathbb C}}
\newcommand{\Z}{{\mathbb Z}}
\newcommand{\R}{{\mathbb R}}

\newcommand{\D}{{\mathbb D}}
\newcommand{\W}{{\mathcal W}}

\newcommand{\hyp}{{\mathbb H}}
\newcommand{\supp}{\operatorname{supp}}

\newcommand{\rank}{\mathop{\rm rank}\nolimits}
\newcommand{\arrow}[1]{\stackrel{#1}{\longrightarrow}}
\newcommand{\Adot}{\mathbf A^\bullet}

\newcommand{\Pdot}{\mathbf P^\bullet}

\newtheorem{defn0}{Definition}[section]
\newtheorem{prop0}[defn0]{Proposition}
\newtheorem{conj0}[defn0]{Conjecture}
\newtheorem{thm0}[defn0]{Theorem}
\newtheorem{lem0}[defn0]{Lemma}
\newtheorem{corollary0}[defn0]{Corollary}
\newtheorem{example0}[defn0]{Example}
\newtheorem{remark0}[defn0]{Remark}
\newtheorem{question0}[defn0]{Question}
\newtheorem{exercise0}[defn0]{Exercise}

\newenvironment{defn}{\begin{defn0}}{\end{defn0}}
\newenvironment{prop}{\begin{prop0}}{\end{prop0}}

\newenvironment{thm}{\begin{thm0}}{\end{thm0}}
\newenvironment{lem}{\begin{lem0}}{\end{lem0}}
\newenvironment{cor}{\begin{corollary0}}{\end{corollary0}}
\newenvironment{exer}{\begin{exercise0}\rm}{\end{exercise0}}
\newenvironment{exm}{\begin{example0}\rm}{\end{example0}}
\newenvironment{rem}{\begin{remark0}\rm}{\end{remark0}}

\newcommand{\thmref}[1]{Theorem~\ref{#1}}
\newcommand{\lemref}[1]{Lemma~\ref{#1}}
\newcommand{\corref}[1]{Corollary~\ref{#1}}
\newcommand{\exref}[1]{Example~\ref{#1}}

\newcommand{\remref}[1]{Remark~\ref{#1}}

\newcommand{\exerref}[1]{Exercise~\ref{#1}}

\newcommand{\mbf}[1]{{\mathbf #1}}

\definecolor{ltblue}{rgb}{0.9,1,1}

\title[Non-isolated Hypersurface Singularities]{Non-isolated Hypersurface Singularities and L\^e Cycles}

\author{David B. Massey}

\date{}

\begin{document}

\maketitle

\begin{abstract} In this series of lectures, I will discuss results for complex hypersurfaces with non-isolated singularities. 

In Lecture 1, I will review basic definitions and results on complex hypersurfaces, and then present classical material on the Milnor fiber and  fibration.  In Lecture 2, I will present basic results from Morse theory, and use them to prove some results about complex hypersurfaces, including a proof of L\^e's attaching result for  Milnor fibers of non-isolated hypersurface singularities. This will include defining the relative polar curve. Lecture 3 will begin with a discussion of intersection cycles for proper intersections inside a complex manifold, and then move on to definitions and basic results  on L\^e cycles and L\^e numbers of non-isolated hypersurface singularities. Lecture 4 will explain the topological importance of L\^e cycles and numbers, and then I will explain, informally, the relationship between the L\^e cycles and the complex of sheaves of vanishing cycles.
\end{abstract}




\section{\bf Lecture 1: Topology of Hypersurfaces and the Milnor fibration} Suppose that $\U$ is an open subset of $\C^{n+1}$; we use $(z_0, \dots, z_n)$ for coordinates. 

Consider a complex analytic (i.e., holomorphic) function $f:\U\rightarrow\C$ which is not locally constant. Then, the {\it hypersurface $V(f)$} defined by $f$ is the purely $n$-dimensional complex analytic space defined by the vanishing of $f$, i.e., 
$$V(f) \ := \ \{\mbf x\in\U \ | \ f(\mbf x)=0\}.$$
 To be assured that $V(f)$ is not empty and to have a convenient point in $V(f)$, one frequently assumes that $\0\in V(f)$, i.e., that $f(\0)=0$. This assumption is frequently included in specifying the function, e.g., we frequently write $f:(\U, \0)\rightarrow (\C, 0)$.

\smallskip

Near each point $\mbf x\in V(f)$, we are interested in the local topology of how $V(f)$ is embedded in $\U$. This is question of how to describe the {\it local, ambient topological-type of $V(f)$} at each point.

\smallskip

A {\it critical point of $f$} is a point $\mbf x\in\U$ at which all of the complex partial derivatives of $f$ vanish. The  {\it critical locus of $f$} is the set of critical points of $f$, and is denoted by $\Sigma f$, i.e.,
$$
\Sigma f \ := \ V\left(\frac{\partial f}{\partial z_0}, \frac{\partial f}{\partial z_1}, \dots, \frac{\partial f}{\partial z_n}\right).
$$

The complex analytic Implicit Function Theorem implies that, if $\mbf x\in V(f)$ and $\mbf x\not\in \Sigma f$, then, in an open neighborhood of $\mbf x$, $V(f)$ is a {\bf complex analytic submanifold} of $\U$; thus, we completely understand the ambient topology of $V(f)$ near a non-critical point. However, if $\mbf x\in\Sigma f$, then it is possible that $V(f)$ is not even a topological submanifold of $\U$ near $\mbf x$.

Note that, as sets, $V(f) = V(f^2)$, and that {\bf every} point of $V(f^2)$ is a critical point of $f^2$. In fact, this type of problem occurs near a point $\mbf p$ any time that an irreducible component of $f$ (in its unique factorization in the unique factorization domain $\mathcal O_{\U, \mbf p}$) is raised to a power greater than one. Hence, when considering the topology of $V(f)$ near a point $\mbf p\in V(f)$, it is standard to assume that $f$ is {\it reduced}, i.e., has no such repeated factors. This is equivalent to assuming that $\dim_\mbf p\Sigma f<n$.

\smallskip

Let's look at a simple, but important, example.

\smallskip

\noindent{\rule{1in}{1pt}}

\begin{exm} Consider $f:(\C^2, \0)\rightarrow (\C, 0)$ given by $f(x,y) = y^2-x^3$. 

\smallskip

It is trivial to check that $\Sigma f=\{\0\}$. Thus, at (near) every point of $V(f)$ other than the origin, $V(f)$ is a complex analytic submanifold of $\C^2$.

     \begin{figure}[h]
\begin{center}
\includegraphics[width=2in]{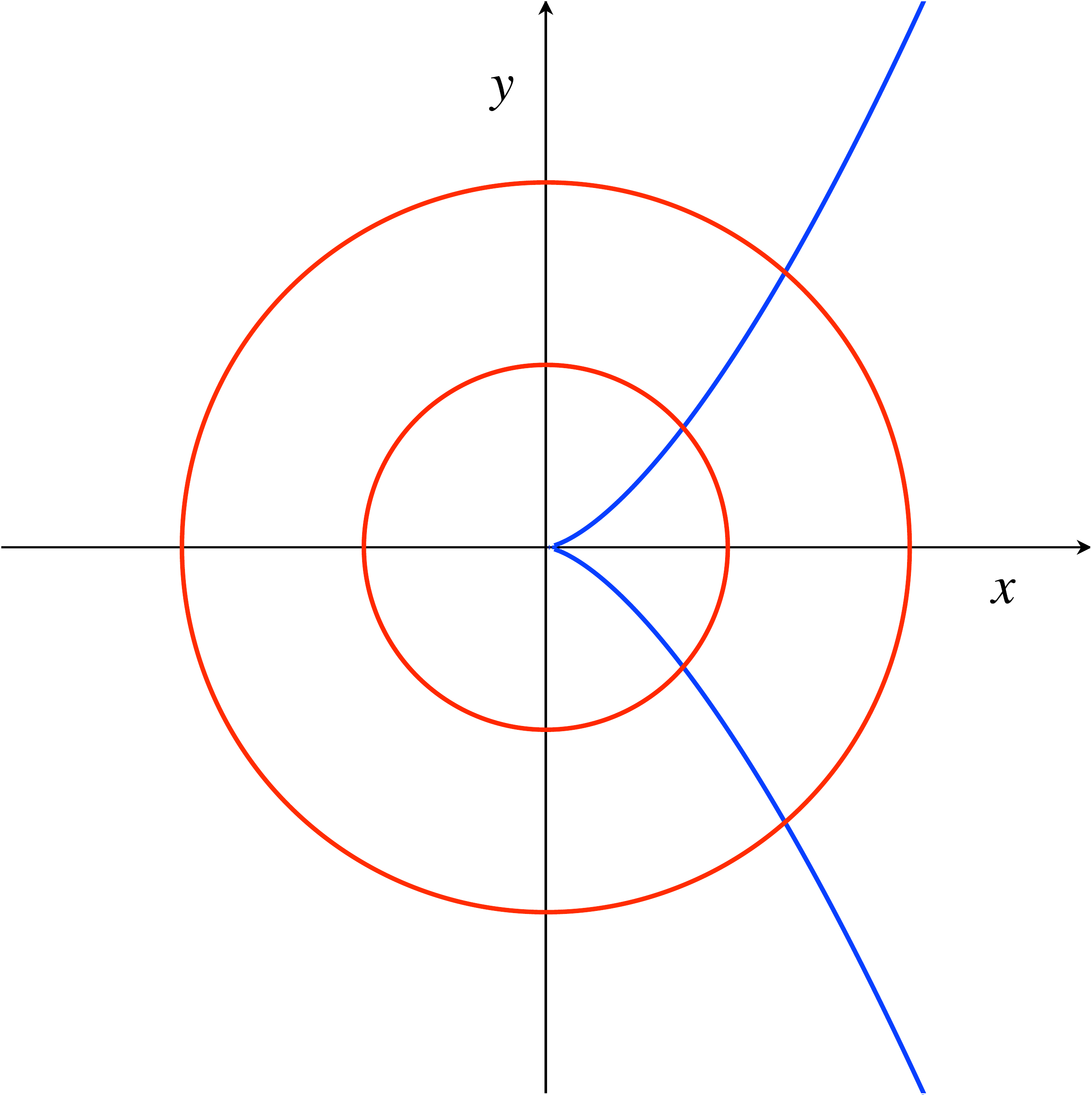}
\captionof{figure}{A cusp, intersected by two ``spheres''.}\label{fig:sphere1}
\end{center}
\end{figure}

In the figure, ignoring for now the two circles, you see the graph of $V(f)$, but drawn over the real numbers. We draw graphs over the real numbers since we can't draw a picture over the complex numbers, but we hope  that the picture over the real numbers gives us some intuition for what happens over the complex numbers.

Note that $f$ has a critical point at the origin, and so the complex analytic Implicit Function Theorem does {\bf not} guarantee that $V(f)$ is a complex submanifold of $\C^2$ near $\0$. If the real picture is not misleading, it appears that $V(f)$ is not even a smooth ($C^\infty$) submanifold of $\C^2$ near $\0$; this is true.

However, the real picture is, in fact, misleading in one important way. Over the real numbers, $V(f)$ is a topological submanifold of $\R^2$ near $\0$, i.e., there exist open neighborhoods $\U$ and $\W$ of the origin in $\R^2$ and a homeomorphism of triples 
$$
(\U, \, \U\cap V(f),\, \0)\cong (\W,\, \W\cap V(x),\, \0).
$$
However, over the complex numbers $V(f)$ is {\bf not}  a topological submanifold of $\C^2$ near $\0$. This takes some work to show.

\smallskip

Why have we drawn the two circles in the figure? Because we want you to observe two things, which correspond to a theorem that we shall state below. First, that the topological-type of the hypersurface seems to stabilize inside open balls of sufficiently small radius, e.g., the hypersurface ``looks'' the same inside the open disk $B^{{}^\circ}_2$ bounded by the bigger circle as it does inside the open disk $B^{{}^\circ}_1$ bounded by the smaller circle; of course, in $\C^2$, ``disk'' becomes ``$4$-dimensional ball'', and ``circle'' becomes ``$3$-dimensional sphere''. Second,  it appears that this ambient topological-type can be obtained by taking the (open) cone on the bounding sphere and its intersection with the hypersurface.
\smallskip

We make all of this precise below.
\end{exm}

\smallskip

\noindent{\rule{1in}{1pt}}

Let us first give rigorous definitions of the {\it local, ambient topological-type of a hypersurface} and of a {\it singular point}.

\begin{defn} Suppose that $\U$ is an open subset of $\C^{n+1}$, and that we have a complex analytic function $f:\U\rightarrow\C$ which is not locally constant. Let $\mbf p\in V(f)$. Then, the {\bf local, ambient topological-type of $V(f)$ at $\mbf p$} is the homeomorphism-type of the germ at $\mbf p$ of the triple $(\U, V(f), \mbf p)$.

In other words, if $g:\W\rightarrow\C$ is another such function, and $\mbf q\in V(g)$, then the local, ambient topological-type of $V(f)$ at $\mbf p$ is the same as that of $V(g)$ at $\mbf q$ if and only if there exist open neighborhoods $\U'$ and $\W'$ of $\mbf p$ and $\mbf q$, respectively, and a homeomorphism of triples
$$
(\U',\, \U'\cap V(f),\, \mbf p) \ \cong \ (\W',\, \W'\cap V(g),\, \mbf q).
$$

\smallskip

The {\bf trivial local, ambient topological-type} is that of $(\C^{n+1},\, V(z_0),\, \0)$. To say that $V(f)$ has the trivial topological-type at a point $\mbf p$ is simply to say that $V(f)$ is a topological submanifold of $\U$ near $\mbf p$.

\smallskip

A point on a hypersurface at which it has the trivial local, ambient topological-type is called a {\bf regular point} of the hypersurface. A non-regular point on a hypersurface is called a {\bf singular point} or a {\bf singularity}. The {\bf set of singular points} of $V(f)$ is denoted by $\Sigma V(f)$.
\end{defn}

\medskip

\begin{rem} You may question our terminology above. Shouldn't ``regular'' and ``singular'' have something to do with smoothness, not just topological data? In fact, it turns out that there is a very strong dichotomy here.

\smallskip

 If $f$ is reduced at $\mbf p$, then, in an open neighborhood of $\mbf p$, $\Sigma f=\Sigma V(f)$. This is not trivial to see, and uses the Curve Selection Lemma (see \lemref{lem:curveselection} in the Appendix) to show that, near a point in $V(f)$, $\Sigma f\subseteq V(f)$, and then uses results on Milnor fibrations. 
 
But, what it implies is that, at a point on a hypersurface, the hypersurface is either an analytic submanifold or is not even a topological submanifold. Therefore, all conceivable notions of ``regular'' and ``singular'' agree for complex hypersurfaces.

This also explains the frequent, mildly bad, habit of using the terms ``critical point of $f$'' and ``singular point of $V(f)$'' interchangeably.
\end{rem}

\medskip

The following theorem can be found in the work of \L ojasiewicz in \cite{lojbooknotes}, and is now a part of the general theory of Whitney stratifications. We state the result for hypersurfaces in affine space, but the general result applies to arbitrary analytic sets. We recall the definition of the cone and the cone on a pair in the Appendix; in particular, recall that the cone on a pair is a triple, which includes the cone point.

\begin{thm} {\rm (\L ojasiewicz, 1965)} Suppose that $\U$ is an open subset of $\C^{n+1}$, and that we have a complex analytic function $f:\U\rightarrow\C$ which is not locally constant. Let $\mbf p\in V(f)$, and for all $\epsilon>0$, let $B_\epsilon(\mbf p)$ and $S_\epsilon(\mbf p)$ denote the closed ball and sphere of radius $\epsilon$, centered at $\mbf p$, in $\C^{n+1}$. Let $B_\epsilon^\circ(\mbf p)$ denote the corresponding open ball.

Then, there exists $\epsilon_0>0$ such that $B_{\epsilon_0}(\mbf p)\subseteq \U$ and such that, if $0<\epsilon\leq\epsilon_0$, then:

\begin{enumerate}
\item $\big(B_\epsilon(\mbf p),\, B_\epsilon(\mbf p)\cap V(f), \mbf p\big)$ is homeomorphic to the triple $c\big(S_\epsilon(\mbf p),\, S_\epsilon(\mbf p)\cap V(f)\big)$ by a homeomorphism which is the ``identity'' on $S_\epsilon(\mbf p)$ when it is identified with $S_\epsilon(\mbf p)\times\{0\}$.

In particular,
$$
\big(B^\circ_\epsilon(\mbf p),\, B^\circ_\epsilon(\mbf p)\cap V(f), \mbf p\big) \ \cong \ c^\circ\big(S_\epsilon(\mbf p),\, S_\epsilon(\mbf p)\cap V(f)\big);
$$
\item the homeomorphism-type of the pair $\big(S_\epsilon(\mbf p),\, S_\epsilon(\mbf p)\cap V(f)\big)$ is independent of the choice of $\epsilon$ (provided $0<\epsilon\leq \epsilon_0$).
\end{enumerate}

Thus, the local, ambient topological-type of $V(f)$ at $\mbf p$ is determined by the homeomorphism-type of the pair $\big(S_\epsilon(\mbf p),\, S_\epsilon(\mbf p)\cap V(f)\big)$, for sufficiently small $\epsilon>0$.
\end{thm}

\medskip

\begin{defn} The space $S_\epsilon(\mbf p)\cap V(f)$ (or its homeomorphism-type) for sufficiently small $\epsilon>0$ is called the {\bf real link of $V(f)$ at $\mbf p$} and is frequently denoted by $K$.
\end{defn}

\medskip

\begin{rem} The letter $K$ is used because, in the first interesting case, of complex curves in $\C^2$, the real link is a knot (or link) in $S^3$, and how this knot is embedded in $S^3$ completely determines the local, ambient topological-type.
\end{rem}

\medskip

\begin{exer}\label{exer:easy} Consider the following examples:

\begin{enumerate}
\item $f:\C^2\rightarrow \C$ given by $f(x,y)=xy$. Show that $V(f)$ is not a topological manifold at $\0$ (and so, is certainly not a topological submanifold).

\smallskip

\item $f:\C^2\rightarrow\C$ given by $f(x,y)=y^2-x^3$. Show that $V(f)$ is homeomorphic to a disk near $\0$, and so is a topological manifold. Now, parameterize $K=S^3_\epsilon\cap V(f)$ and show that you obtain the trefoil knot in $S^3_\epsilon$. Conclude that $V(f)$ is not a topological submanifold of $\C^2$ near $\0$.

\smallskip

\item $f:\C^3\rightarrow\C$ given by $f(x,y,z)=2xy-z^2$. Show that $K$ is homeomorphic to real projective $3$-space. Conclude that $V(f)$ is not a topological manifold near $\0$. (Hint: Use $x=s^2$, $y=t^2$, and $z=\sqrt{2}st$, and note that a point of $V(f)$ is not represented by a unique choice of $(s,t)$.)
\end{enumerate}
\end{exer}

\smallskip

\noindent\rule{1in}{1pt}

As you can probably tell, the functions used in \exerref{exer:easy} were chosen very specially, and, in general, it is unreasonable to expect to analyze the topology of a hypersurface at a singular point via such concrete unsophisticated techniques.

\smallskip

So...how does one go about understanding how the real link $K$ embeds in a small sphere? 

\smallskip

One large piece of data that one can associate to this situation is the topology of the {\bf complement}. This, of course, is not complete data about the embedding, but it is a significant amount of data.

\smallskip

For ease of notation,  assume that we have a complex analytic function $f:(\U, \0)\rightarrow (\C,0)$ which is not locally constant, and that we wish to understand the local, ambient topology of $V(f)$ at $\0$. We will suppress the references to the center $\mbf p=\0$ in our notation for spheres and balls. So, how do you analyze $S_\epsilon-S_\epsilon\cap V(f) = S_\epsilon-K$ for sufficiently small $\epsilon>0$? 

\smallskip

Milnor gave us many tools in his 1968 book \cite{milnorsing}. He proved that, for sufficiently small $\epsilon>0$, the map 
	$$
	\frac{f}{|f|}:S_\epsilon-K\rightarrow S^1\subseteq\C
	$$
is a smooth, locally trivial fibration, and then proved many results about the fiber. (To review what a smooth, locally trivial fibration is, see the Appendix.)

\smallskip

We will state some of the results of Milnor and others about the above fibration, which is now known as the {\it Milnor fibration}. Below, $\D_\delta$ denotes a disk in $\C$, centered at the origin, of radius $\delta$, and so $\partial \D_\delta$ is its boundary circle.

The following theorem is a combination of Theorem 4.8 and Theorem 5.11 of \cite{milnorsing}, together with Theorem 1.1 of \cite{relmono}.

\begin{thm} (Milnor, 1968 and L\^e, 1976) Suppose that $f:(\U, \0)\rightarrow (\C,0)$ is a complex analytic function. Then, there exists $\epsilon_0>0$ such that, for all $\epsilon$ with $0<\epsilon\leq\epsilon_0$, there exists $\delta_\epsilon>0$, such that, for all $\delta$ with $0<\delta\leq\delta_\epsilon$, the map $f/|f|$ from $S_\epsilon-S_\epsilon\cap V(f)=S_\epsilon-K$ to $S^1$ is a smooth locally trivial fibration.

Furthermore, this smooth locally trivial fibration is diffeomorphic to the restriction  
$$f:B^\circ_\epsilon\cap f^{-1}(\partial \D_\delta)\rightarrow \partial \D_\delta.$$

Finally,  the restriction 
$$f:B_\epsilon\cap f^{-1}(\partial \D_\delta)\rightarrow \partial \D_\delta$$
 (note the closed ball) is a smooth locally trivial fibration, in which the fiber is a smooth manifold with boundary. This fibration is fiber-homotopy-equivalent to the one using the open ball (i.e., is isomorphic up to homotopy).
\end{thm}

\bigskip

\begin{rem}\label{rem:contract} It will be important to us later that Milnor's proof of the above theorem also shows that
$B_\epsilon\cap f^{-1}(\D_\delta)$ is homeomorphic to $B_\epsilon$ and, hence, is contractible. This is sometimes referred to as a {\bf Milnor tube}.
\end{rem}

\bigskip

\begin{defn} Either one of the first two isomorphic fibrations given in the definition above is called the {\bf Milnor fibration} of $f$ at $\0$, and the corresponding fiber is called the {\bf Milnor fiber}.

The third and final fibration from the theorem above is called the {\bf compact Milnor fibration}, and the corresponding fiber is called the {\bf compact Milnor fiber}.

If we are interested in the Milnor fibration and/or Milnor fiber only up to homotopy, then any of the three fibrations and fibers are called the {\bf Milnor fibration} and {\bf Milnor fiber}.
\end{defn}

\bigskip

\begin{rem} As $\U$ is an open subset of $\C^{n+1}$, the Milnor fiber is a complex $n$-manifold, and so is a real $2n$-manifold. The compact Milnor fiber is thus a compact real $2n$-manifold with boundary.

We should also remark that the Milnor fibration exists at each point $\mbf p\in V(f)$; one simply replaces the ball and spheres centered at $\0$ with balls and spheres centered at $\mbf p$.

As a final remark, we should mention that the phrase ``there exists $\epsilon_0>0$ such that, for all $\epsilon$ with $0<\epsilon\leq\epsilon_0$, there exists $\delta_\epsilon>0$, such that, for all $\delta$ with $0<\delta\leq\delta_\epsilon$'' is usually abbreviated by writing simply ``For $0<\delta\ll\epsilon\ll1$''. This is read aloud as  ``for all sufficiently small positive $\epsilon$, for all sufficiently small positive $\delta$ (small compared to the choice of $\epsilon$)''.
\end{rem}

\bigskip

We will now list a number of results on the Milnor fibration and Milnor fiber. Below, we let $\U$ be an open neighborhood of the origin in $\C^{n+1}$, $f:(\U,\0)\rightarrow (\C,0)$ is a complex analytic function, $F_{f, \0}$ denotes the Milnor fiber of $f$ at $\0$, and we let $s:=\dim_\0\Sigma f$.

\begin{enumerate}

\item If $\0\not\in \Sigma f$, then $F_{f,\0}$ is diffeomorphic to a ball and so, in particular, is contractible and has trivial homology (i.e., the homology of a point).

\medskip

\item $F_{f,\0}$ has the homotopy-type of a finite $n$-dimensional CW-complex. In particular, if $k>n$, then the homology $H_k(F_{f,\0}; \Z)=0$, and $H_n(F_{f,\0}; \Z)$ is free Abelian. (See \cite{milnorsing}, Theorem 5.1.)

\medskip

\item $F_{f,\0}$ is $(n-s-1)$-connected. (For $s=0$, see \cite{milnorsing}, Lemma 6.4. For general $s$, see \cite{katomatsu}.)

\medskip

\item Suppose that $s=0$. Then Items 1 and 2 imply that $F_{f,\0}$ has the homotopy-type of the one-point union of a finite collection of $n$-spheres; this is usually referred to as a {\bf bouquet of spheres}. The number of spheres in the bouquet, i.e., the rank of $H_n(F_{f,\0}; \Z)$, is called the {\bf Milnor number of $f$ at $\0$} and is denoted by either $\mu_f(\0)$ or $\mu_\0(f)$.

\medskip

\item The Milnor number of $f$ at an isolated critical point can be calculated algebraically by taking the complex dimension of the {\bf Jacobian algebra}, i.e.,
$$
\mu_f(\0) \ = \ \dim_{\C}\,\frac{\C\{z_0,\dots, z_n\}}{\left\langle\frac{\partial f}{\partial z_0}, \dots, \frac{\partial f}{\partial z_n}\right\rangle},
$$
where $\C\{z_0,\dots, z_n\}$ is the ring of convergent power series at the origin. (This follows at once from \cite{milnorsing}, Theorem 7.2, by using a result of V. Palamodov in \cite{palamodov}.)

In particular, if $s=0$, then $\mu_f(\0)>0$ if and only if $\0\in\Sigma f$.   

\medskip

\item In Lemma 9.4 of \cite{milnorsing}, Milnor proves that, if $f$ is a weighted homogeneous polynomial, then the Milnor fiber of $f$ at $\0$ is diffeomorphic to the global fiber $f^{-1}(1)$ in $\C^{n+1}$.

\medskip

\item If $f: (\U, \0) \rightarrow (\C, 0)$ and $g: (\U', \0) \rightarrow (\C, 0)$ are analytic functions, then the Milnor fibre of the function $h: (\U\times\U', \0) \rightarrow (\C, 0)$ defined by $h(\mathbf w, \mathbf z) := f(\mathbf w) +
g(\mathbf z)$ is homotopy-equivalent to the join (see the Appendix), $F_{f, \0} * F_{g,
\0}$, of the Milnor fibres of $f$ and $g$.  

This determines the homology of $F_{h, \0}$ in a simple way, since the reduced homology of the join of
two spaces $X$ and
$Y$ is given by 
$$\widetilde H_{j+1}(X*Y) = \sum_{k + l = j}\widetilde H_k(X)\otimes \widetilde H_l(Y) \oplus\hskip -0.1in  \sum_{k + l = j-1}\hskip -0.1in\operatorname{Tor}\left(\widetilde H_k(X), \widetilde H_l(Y)\right),$$ 
where all homology groups are with $\Z$ coefficients.

This is the Sebastiani-Thom Theorem, proved in different cases by many people. See, for instance, \cite{sebthom}, \cite{okasebthom}, \cite{sakamoto}, \cite{nemethisebthom1},  \cite{nemethisebthom2}, and \cite{masseysebthom}.

\medskip

\item Let $\U$ and $\W$ be open neighborhoods  of $\0$ in $\C^{n+1}$, let $f:(\U,\0)\rightarrow(\C, 0)$ and $g:(\U, \0)\rightarrow(\C, 0)$ be reduced complex analytic functions which define hypersurfaces with the same ambient topological-type at the origin. Then, there exists a homotopy-equivalence $\alpha:F_{f, \0}\rightarrow F_{g, \0}$ such that the induced isomorphism on homology commutes with the respective Milnor monodromy automorphisms.

In particular, the homotopy-type of the Milnor fiber of a reduced complex analytic function $f$ is an invariant of the local, ambient topological-type of $V(f)$, and so, for hypersurfaces defined by a reduced function with an isolated critical point, the Milnor number is an invariant of the local, ambient topological-type. 

(For $s=0$, this result appears in a remark of Teissier in \cite{teissierdeform} in 1972 and in \cite{teissiercargese} in 1973. The general result, with a monodromy statement, is due to L\^e in  \cite{topsing} and \cite{leattach}, which both appeared in 1973.)

\medskip

\item Suppose that $\0\in\Sigma f$. Let $T^i_{f,\0}: H_i(F_{f, \0};\Z)\rightarrow H_i(F_{f, \0};\Z)$ denote the monodromy automorphism in degree $i$. Then, the Lefschetz number of the monodromy $T^*_{f,\0}$ is zero, i.e.,
$$
\sum_i (-1)^i\operatorname{trace}\left(T^i_{f,\0}\right) \ = \ 0. 
$$
(See \cite{acamp}.)

\medskip

\item The previous item implies that the converse to Item 1 is true. Thus, the Milnor fiber $F_{f, \0}$ has trivial homology (i.e., has the homology of a point) if and only if  $\0\not\in \Sigma f$ (and so, in particular, $V(f)$ is a topological submanifold of affine space at $\0$).

\end{enumerate}

\bigskip

\begin{exer}\label{exer:converse} In some/many cases, the Milnor number can be calculated by hand.

\begin{enumerate}

\item Calculate the Milnor number at $\0$ of $f(x,y)=y^2-x^3$, which defines a cusp.

\medskip

\item Calculate the Milnor number at $\0$ of $f(x,y)=y^2-x^3-x^2$, which defines a node. Conclude that the node and cusp have different ambient topological types.

\medskip

\item Show that $f=y^2-x^5$ and $g=y^3-x^3$ both have Milnor number $4$ at the origin, but do not define hypersurfaces with the same ambient topological-type at the origin (actually, these hypersurfaces do not have the same topological-type at the origin, leaving out the term ``ambient'').

Thus, even for isolated critical points, the converse of Item 8, above, is false.
\end{enumerate}

\end{exer}

\medskip

\begin{exer} In special cases, one can calculate the homology groups of the Milnor fiber of a non-isolated critical point. Consider $f(x,y,t)=y^2-x^3-tx^2$. Show that $\dim_\0\Sigma f=1$ and calculate the homology groups of $F_{f,\0}$. (Hint: Use the Sebastiani-Thom Theorem. Also, use Milnor's result for weighted homogeneous polynomials, and that the homotopy-type of the Milnor fiber is certainly invariant under local analytic coordinate changes.)

The function $f$ can be thought of as a family of hypersurfaces,  parameterized by $t$, where each member of the family has an isolated critical point at the origin; so, this is usually described as a {\bf family of nodes which degenerates to a cusp at $t=0$}.
\end{exer}

\medskip

Despite Item 3 of \exerref{exer:converse}, the stunning conclusion of L\^e and Ramanujam is that the converse of Item 7, above, {\bf is} true in the case of isolated critical points if $f$ and $g$ are in the same analytic family (with one dimension restriction):

\begin{thm} \textnormal{(L\^e-Ramanujam, \cite{leramanujam})}
Suppose $n\neq 2$, and $f$ and $g$ are part of an analytic family of functions with isolated critical points, all of which have the same Milnor number, then $f$ and $g$ define hypersurfaces with the same local, ambient topological-type.
\end{thm}

\smallskip

\noindent{\rule{1in}{1pt}}

Thus, for hypersurfaces with isolated singularities, the Milnor number is algebraically calculable, determines the homology of the Milnor fiber, and its constancy in a family (with one dimension restriction) controls the local ambient topology in the family.

We would like similar data for hypersurfaces with non-isolated singularities. The L\^e numbers succeed at generalizing the Milnor number in many ways, but do not yield such strong results. We shall discuss L\^e cycles and L\^e numbers in the third lecture.

\medskip

In the second lecture, we will discuss the basics of Morse Theory, and use it to prove an important result of L\^e from \cite{leattach} on the homology of the Milnor fiber for non-isolated hypersurface singularities.

\bigskip

\section{\bf Lecture 2: Morse Theory, the relative polar curve, and two applications} 

Many of the results in \cite{milnorsing} are proved using {\it Morse Theory}, and so we wish to give a quick introduction to the subject. We will then give some examples of how Morse Theory is used in the study of singular hypersurfaces.

\bigskip

Morse Theory is the study of what happens at the most basic type of critical point of a smooth map. The classic, beautiful  references for Morse Theory are \cite{milnormorse} and \cite{milnorhcobord}. We also recommend the excellent, new introductory treatment in \cite{intromorse}.

\bigskip

In this section, until we explicitly state otherwise, $f:N\rightarrow\R$ will be a smooth function from a smooth manifold of dimension $n$ into $\R$. For all $a\in \R$, let $N_{\leq a}:=f^{-1}((-\infty, a])$. Note that if $a$ is a regular value of $f$, then $N_{\leq a}$ is a {\it smooth manifold with boundary $\partial N_{\leq a}= f^{-1}(a)$\/} (see, for instance, \cite{spivak1}).

\smallskip

The following is essentially Theorem 3.1 of \cite{milnormorse}.

\smallskip

\begin{thm}\label{thm:nocrit} Suppose that $a, b\in\R$ and $a<b$. Suppose that $f^{-1}([a,b])$ is compact and contains no critical points of $f$. 

Then, the restriction $f: f^{-1}([a,b])\rightarrow [a,b]$ is a trivial fibration, and  $N_{\leq a}$ is a deformation retract of $N_{\leq b}$ via a smooth isotopy. In particular, $N_{\leq a}$ is diffeomorphic to $N_{\leq b}$.
\end{thm}

\bigskip

Now, let $\mbf p\in N$, and let $(x_1, ..., x_n)$ be a smooth, local coordinate system for $N$ in an open neighborhood of $\mbf p$.

\smallskip

\begin{defn} The point $\mbf p$ is a {\bf non-degenerate critical point of $f$} provided that $\mbf p$ is a critical point of $f$, and that the Hessian matrix $\left(\frac{\partial^2f}{\partial x_i\partial x_j}(\mbf p)\right)_{i, j}$ is non-singular.

The {\bf index of $f$ at a non-degenerate critical point $\mbf p$} is the number of negative eigenvalues of $\left(\frac{\partial^2f}{\partial x_i\partial x_j}(p)\right)_{i, j}$, counted with multiplicity.
\end{defn}

\smallskip

Note that since the Hessian matrix is a real symmetric matrix, it is diagonalizable and, hence, the algebraic and geometric multiplicities of eigenvalues are the same. 

\medskip

\begin{exer}
Prove that $\mbf p$ being a non-degenerate critical point of $f$ is independent of the choice of local coordinates on $N$. 

The index of $f$ at a non-degenerate critical point $\mbf p$ can also can characterized as the index of the bilinear form $B$ defined by the Hessian matrix; this is defined to be the dimension of a maximal subspace on which $B$ is negative-definite. Using this, prove that the index of $f$ at a non-degenerate critical is also independent of the coordinate choice.\end{exer}

\smallskip

The following is Lemma 2.2 of \cite{milnormorse}, which tells us the basic structure of $f$ near a non-degenerate critical point.

\begin{lem}\label{lem:morse}{\rm (The Morse Lemma)} Let $\mbf p$ be a non-degenerate critical point of $f$. Then, there is a local coordinate system $(y_1, \dots, y_n)$ in an open neighborhood $\U$ of $\mbf p$, with $y_i(\mbf p)=0$, for all $i$, and such that, for all  $\mbf x\in\U$,
$$
f(\mbf x)=f(\mbf p)-(y_1(\mbf x))^2-(y_2(\mbf x))^2-\dots-(y_\lambda(\mbf x))^2+(y_{\lambda+1}(\mbf x))^2+\dots+(y_n(\mbf x))^2,
$$
where $\lambda$ is the index of $f$ at $\mbf p$.

In particular, the point $\mbf p$ is an isolated critical point of $f$.
\end{lem}

\bigskip

The fundamental result of Morse Theory is a description of how $N_{\leq b}$ is obtained from $N_{\leq a}$, where $a<b$, and where $f^{-1}([a, b])$ is compact and contains a single critical point of $f$, and that critical point is contained in $f^{-1}((a, b))$ and is non-degenerate. See \cite{milnormorse}.

Recall that ``attaching a $\lambda$-cell to a space $X$'' means taking a closed ball of dimension $\lambda$, and attaching it to $X$ by identifying points on the boundary of the ball with points in $X$.

\smallskip

\begin{thm} \label{thm:morse} Suppose that $a<b$,  $f^{-1}([a, b])$ is compact and contains exactly one critical point of $f$, and that this critical point is contained in $f^{-1}((a, b))$ and is non-degenerate of index $\lambda$. 

Then, $N_{\leq b}$  has the homotopy-type of $N_{\leq a}$  with a $\lambda$-cell attached, and so $H_i(N_{\leq b}, N_{\leq a};\ \Z)=0$ if $i\neq \lambda$, and $H_\lambda(N_{\leq b}, N_{\leq a};\ \Z)\cong \Z$.
\end{thm}

Thus, functions $f:N\rightarrow \R$ that have only non-degenerate critical points are of great interest, and so we make a definition.

\begin{defn}\label{def:morse} The smooth function $f:N\rightarrow\R$ is a {\bf Morse function} if and only if all of the critical points of $f$ are non-degenerate.
\end{defn}

\medskip

Definition \ref{def:morse}  would not be terribly useful if there were very few Morse functions. However, there are a number of theorems which tell us that Morse functions are very plentiful. We remind the reader that ``almost all'' means except for a set of measure zero.

\begin{thm}\label{thm:linearmorse} {\rm (\cite{milnorhcobord}, p. 11)} If $g$ is a $C^2$ function from an open subset $\U$ of $\R^n$ to $\R$, then, for almost all linear functions $L:\R^n\rightarrow \R$, the function $g+L:\U\rightarrow\R$ is a Morse function.
\end{thm}

\begin{thm}\label{thm:morseembed}{\rm (\cite{milnormorse}, Theorem 6.6)} Let $M$ be a smooth submanifold of $\R^n$, which is a closed subset of $\R^n$. For all $\mbf p\in\R^n$, let $L_{\mbf p}:M\rightarrow \R$ be given by $L_{\mbf p}(\mbf x):=||\mbf x-\mbf p||^2$. Then, for almost all $\mbf p\in\R^n$, $L_{\mbf p}$ is a proper Morse function such that $M_{\leq a}$ is compact for all $a$.
\end{thm}

\begin{cor}\label{cor:cw} {\rm (\cite{milnormorse}, p. 36)} Every smooth manifold $M$ possesses a Morse function $g:M\rightarrow\R$ such that $M_{\leq a}$ is compact for all $a\in\R$. Given such a function $g$, $M$ has the homotopy-type of a CW-complex  with one cell of dimension $\lambda$ for each critical point of $g$ of index $\lambda$.
\end{cor}

While we stated the above as a corollary to \thmref{thm:morseembed}, it also strongly uses two other results: Theorem 3.5 of \cite{milnormorse} and Whitney's Embedding Theorem, which tells us that any smooth manifold can be smoothly embedded as a closed subset of some Euclidean space.

\smallskip

\noindent{\rule{1in}{1pt}}

We now wish to mention a few complex analytic results which are of importance.

\begin{thm} {\rm (\cite{milnormorse}, p. 39-41)} Suppose that $M$ is an $m$-dimensional complex analytic submanifold of $\C^n$. For all $\mbf p\in\C^n$, let $L_{\mbf p}:M\rightarrow \R$ be given by $L_{\mbf p}(\mbf x):=||\mbf x-\mbf p||^2$. If $\mbf x\in M$ is a non-degenerate critical point of $L_{\mbf p}$, then the index of $L_{\mbf p}$ at $\mbf x$ is less than or equal to $m$.
\end{thm}

\corref{cor:cw} immediately implies:

\begin{cor}  {\rm (\cite{milnormorse}, Theorem 7.2)} If $M$ is an $m$-dimensional complex analytic submanifold of $\C^n$, which is a closed subset of $\C^n$, then $M$ has the homotopy-type of an $m$-dimensional CW-complex. In particular, $H_i(M;\ \Z) = 0$ for $i>m$.
\end{cor}

Note that this result should {\bf not} be considered obvious; $m$ is the {\bf complex} dimension of $M$. Over the real numbers, $M$ is $2m$-dimensional, and so $m$ is frequently referred to as the {\bf middle dimension}. Thus, the above corollary says that the homology of a complex analytic submanifold of $\C^n$, which is closed in $\C^n$, has trivial homology above the middle dimension.

The reader might hope that the corollary above would allow one to obtain nice results about compact complex manifolds; this is {\bf not} the case. The maximum modulus principle, applied to the coordinate functions on $\C^n$, implies that the only compact, connected, complex submanifold of $\C^n$ is a point.

\bigskip

Suppose now that $M$ is a connected complex $m$-manifold, and that $c:M\rightarrow \C$ is a complex analytic function. Let $\mbf p\in M$, and let $(z_1, ..., z_m)$ be a complex analytic coordinate system for $M$ in an open neighborhood of $\mbf p$.

\smallskip

Analogous to our definition in the smooth case, we have:

\smallskip

\begin{defn} The point $\mbf p$ is a {\bf complex non-degenerate critical point of $c$} provided that $\mbf p$ is a critical point of $c$, and that the Hessian matrix $\left(\frac{\partial^2c}{\partial z_i\partial z_j}(\mbf p)\right)_{i, j}$ is non-singular.
\end{defn}

There is a complex analytic version of the Morse Lemma, \lemref{lem:morse}:

\begin{lem} Let $\mbf p$ be a complex non-degenerate critical point of $c$. Then, there is a local complex analytic coordinate system $(y_1, \dots, y_m)$ in an open neighborhood $\U$ of $\mbf p$, with $y_i(\mbf p)=0$, for all $i$, and such that, for all  $\mbf x\in\U$,
$$
c(\mbf x)=c(\mbf p)+(y_1(\mbf x))^2+(y_2(\mbf x))^2+\dots+(y_m(\mbf x))^2.
$$

In particular, the point $\mbf p$ is an isolated critical point of $c$.
\end{lem}

\medskip

\begin{prop} The map $c$ has a complex non-degenerate critical point at $\mbf p$ if and only if $c-c(\mbf p)$ has an isolated critical point at $\mbf p$ and the Milnor number $\mu_{c-c(\mbf p)}(\mbf p)$ equals $1$.
\end{prop}

\medskip

The first statement of the  following theorem is proved in exactly the same manner as \thmref{thm:linearmorse}; one uses the open mapping principle for complex analytic functions to obtain the second statement.

\begin{thm} If $c$ is a complex analytic function from an open subset $\U$ of $\C^m$ to $\C$, then, for almost all complex linear functions $L:\C^m\rightarrow \C$, the function $c+L:\U\rightarrow\C$ has no complex degenerate critical points. 

In addition, for all $\mbf x\in\U$, there exists an open, dense subset $\W$ in  ${\operatorname{Hom}}_{\C}(\C^m, \C)\cong \C^m$ such that, for all $L\in\W$, there exists an open neighborhood $\U^\prime\subseteq\U$ of $\mbf x$ such that $c+L$ has no complex degenerate critical points in $\U^\prime$.
\end{thm}

\bigskip

Finally, we leave the following result as an exercise for the reader. We denote the real and imaginary parts of $c$ by $\operatorname{Re} c$ and  $\operatorname{Im} c$, respectively.

\begin{exer} Show that:
\begin{enumerate} 
\item
$$
\Sigma c \ = \ \Sigma(\operatorname{Re} c) \ = \ \Sigma(\operatorname{Im} c)
$$
and that, if $c(\mbf p)\neq 0$, then $\mbf p\in \Sigma c$ if and only if $\mbf p\in \Sigma\left(|c|^2\right)$.

\medskip

\item Suppose that $\mbf p$ is a complex non-degenerate critical point of $c$. Prove that the real functions $\operatorname{Re} c: M\rightarrow \R$ and $\operatorname{Im} c: M\rightarrow \R$ each have a (real, smooth)  non-degenerate critical point at $\mbf p$ of index precisely equal to $m$, the complex dimension of $M$. 

In addition, if $c(\mbf p)\neq 0$, then prove that the real function $|c|^2: M\rightarrow \R$ also has a non-degenerate critical point of index $m$ at $\mbf p$.

\end{enumerate}

\end{exer}

\smallskip

\noindent{\rule{1in}{1pt}}

Now we want to use Morse Theory to sketch the proofs of two important results: one due to Milnor (see \cite{milnorsing} Theorems 6.5 and 7.2, but our statement and proof are different) and one due to L\^e \cite{leattach}.

\medskip

First, it will be convenient to define the {\bf relative polar curve} of Hamm, L\^e, and Teissier; see  \cite{hammlezariski} and \cite{teissiercargese}. Later, we will give the relative polar curve a cycle structure, but -- for now -- we give the classical definition as a (reduced) analytic set.

Suppose that $\U$ is an open subset of $\C^{n+1}$ and that $f:(\U,\0)\rightarrow (\C, 0)$ is a complex analytic function which is not locally constant. Let $L$ denote a non-zero linear form, and let $\Sigma (f, L)$ denote the critical locus of the map \hbox{$(f,L):\U\rightarrow \C^2$}.

\medskip

\begin{thm}\label{thm:relpolar} \textnormal{(\cite{hammlezariski}, \cite{teissiercargese})} For a generic choice of $L$:

\begin{enumerate}
\item the analytic set 
$$
\Gamma_{f, L} \ := \ \overline{\Sigma (f, L)-\Sigma f}
$$
is purely $1$-dimensional at the origin (this allows for the case where $\0\not\in\Gamma_{f,L}$); 

\bigskip

\item $\dim_\0\Gamma_{f, L}\cap V(L)\leq 0$ and $\dim_\0\Gamma_{f, L}\cap V(f)\leq 0$ (the $<0$ cases allow for $\Gamma_{f,L}=\emptyset$);

\bigskip

\item for each $1$-dimensional irreducible component $C$ of $\Gamma_{f, L}$ which contains the origin, for $\mbf p\in C-\{\0\}$, close enough to the origin, $f_{|_{V(L-L(\mbf p))}}$ has an isolated critical point at $\mbf p$, and 
$$
\mu_{\mbf p}\left(f_{|_{V(L-L(\mbf p))}}\right) =1.
$$
\end{enumerate}

\end{thm}

\medskip

\begin{exer}\label{exer:polarcurve} Show that $\dim_\0\Gamma_{f, L}\cap V(f)\leq 0$ if and only if $\dim_\0\Gamma_{f, L}\cap V(L)\leq 0$, and that these equivalent conditions imply that $\Gamma_{f, L}$ is purely $1$-dimensional at $\0$. (Hint: Give yourself a coordinate system on $\U$ that has $L$ as one of its coordinates, and parameterize the components of $\Gamma_{f, L}$.)
\end{exer}

\medskip

\begin{rem} Suppose that we choose (re-choose) our coordinate system $(z_0, \dots, z_n)$ for $\C^{n+1}$ so that $z_0=L$. Then, we may consider the scheme
$$
V\left(\frac{\partial f}{\partial z_1}, \dots, \frac{\partial f}{\partial z_n}\right).
$$
Then $\Gamma_{f, L}$ consists of those irreducible components of this scheme which are {\bf not} contained in $\Sigma f$. The condition that $\mu_{\mbf p}\left(f_{|_{V(L-L(\mbf p))}}\right) =1$ in Item 3 of \thmref{thm:relpolar} is equivalent to saying that these irreducible components of the scheme are reduced at points other than the origin.

\medskip

\end{rem}

\bigskip

In light of \thmref{thm:relpolar} and \exerref{exer:polarcurve}, we make the following definition.

\begin{defn} If $L$ is generic enough so that $\Gamma_{f, L}$ is purely $1$-dimensional at $\0$, then we refer to $\Gamma_{f, L}$ as the {\bf relative polar curve} (of $f$ with respect to $L$ at the origin), and denote it by $\Gamma_{f, L}^1$ (note the superscript by the dimension). In this case, we say that {\bf the relative polar curve exists} or, simply, that $\Gamma_{f, L}^1$ exists.

If $\dim_\0\Gamma_{f, L}\cap V(f)\leq 0$ (so that, in particular, $\Gamma_{f, L}^1$ exists), then we say that $V(L)$ (or $L$ itself)  is a {\bf Thom slice} (for $f$ at $\0$).

If Item 3 of \thmref{thm:relpolar} holds, then we say that the relative polar curve is {\bf reduced}.
\end{defn}

\medskip

\begin{exer} Suppose that $\dim_\0\Sigma f=0$. Conclude that
$$
V\left(\frac{\partial f}{\partial z_1}, \dots, \frac{\partial f}{\partial z_n}\right)
$$
is a purely $1$-dimensional local complete intersection, equal to $\Gamma_{f, z_0}^1$ and that $$\dim_\0 \Gamma_{f, z_0}^1\cap V\left(\frac{\partial f}{\partial z_0}\right)=0.$$

If you are familiar with intersection numbers, conclude also that
$$
\mu_f(\0) \ = \ \dim_{\C}\,\frac{\C\{z_0,\dots, z_n\}}{\left\langle\frac{\partial f}{\partial z_0}, \dots, \frac{\partial f}{\partial z_n}\right\rangle} \ = \ \left(\Gamma_{f, z_0}^1\cdot V\left(\frac{\partial f}{\partial z_0}\right)\right)_\0.
$$
\end{exer}

\bigskip

Now we're ready to prove, modulo many technical details, two important results on Milnor fibers.

\medskip

First, the classic result of Milnor:

\begin{thm} \textnormal{(\cite{milnorsing}, Theorem 6.5 and 7.2)} Suppose that $\dim_\0\Sigma f=0$.  Then, the Milnor fiber $F_{f,\0}$ is homotopy-equivalent to a bouquet of $n$-spheres, and the number of spheres in the bouquet is 
$$
\dim_{\C}\,\frac{\C\{z_0,\dots, z_n\}}{\left\langle\frac{\partial f}{\partial z_0}, \dots, \frac{\partial f}{\partial z_n}\right\rangle} \ = \ \left(\Gamma_{f, z_0}^1\cdot V\left(\frac{\partial f}{\partial z_0}\right)\right)_\0.
$$
\end{thm}

\begin{proof} We sketch a proof.

\medskip

Recall, from \remref{rem:contract}, that the Milnor tube $T:=B_\epsilon \cap f^{-1}(\D_\delta)$, for $0<\delta\ll\epsilon\ll 1$, is contractible. Select a complex number $a$ such that $0<|a|\ll\delta$. We wish to show that $T$ is obtained from $F_{f,\0}=T\cap f^{-1}(a)$, up to homotopy, by attaching $\left(\Gamma_{f, z_0}^1\cdot V\left(\frac{\partial f}{\partial z_0}\right)\right)_\0$ $(n+1)$-cells.

The number $a$ is a regular value of $f$ restricted to the compact manifold with boundary $B_\epsilon$, i.e., a regular value when restricted to the open ball and when restricted to the bounding sphere. Consequently, for $0<\eta\ll |a|$, the closed disk, $\D_\eta(a)$, of radius $\eta$, centered at $a$, consists of regular values, and so the restriction of $f$ to a map from $T\cap f^{-1}(\D_\eta(a))$ to $\D_\eta(a)$ is a trivial fibration; in particular, $T\cap f^{-1}(\D_\eta(a))$ is homotopy-equivalent to $F_{f,\0}$. Furthermore, $T$ is diffeomorphic to $T':=B_\epsilon \cap f^{-1}(\D_\delta(a))$ and, hence, $T'$ is contractible. 

\smallskip

We wish to apply Morse Theory to $|f-a|^2$ as its value grows from $\eta$ to $\delta$. However, there is no reason for the critical points of $f-a$ to be complex non-degenerate. Thus, we assume that the coordinate $z_0$ is chosen to be a generic linear form and, for $0< |t|\ll \eta$, we consider the map $r:=|f-tz_0-a|^2$ as a map from $B_\epsilon$ to $\R$.

Then, $r^{-1}[0,\eta]$ is homotopy-equivalent to $F_{f,0}$, $T'':=r^{-1}[0, \delta]$ is contractible, $r$ has no critical points on $S_\epsilon$, and all of the critical points of $f-tz_0-a$ in $B_\epsilon^\circ$ are complex non-degenerate. Consequently, complex Morse Theory tells us that the contractible set $T''$ is constructed by attaching $(n+1)$-cells to $F_{f, \0}$. Hence, $F_{f,0}$ has the homotopy-type of a finite bouquet of $n$-spheres.

How many $n$-spheres are there? One for each critical point of $f-tz_0-a$ in $B_\epsilon^\circ$. Therefore, the number of $n$-spheres in the homotopy-type is the number of points in
$$
B_\epsilon^\circ\cap V\left(\frac{\partial f}{\partial z_0}-t, \frac{\partial f}{\partial z_1}, \dots, \frac{\partial f}{\partial z_n}\right) \ = \ B_\epsilon^\circ\cap \Gamma_{f, z_0}^1\cap V\left(\frac{\partial f}{\partial z_0}-t\right),
$$
where $0< |t|\ll \epsilon\ll 1$.
This is precisely the intersection number
$$\left(\Gamma_{f, z_0}^1\cdot V\left(\frac{\partial f}{\partial z_0}\right)\right)_\0.$$
\end{proof}

\bigskip

Now we wish to sketch the proof of L\^e's main result of \cite{leattach} for hypersurface singularities of arbitrary dimension. 

We continue to assume that $\U$ is an open subset of $\C^{n+1}$ and that $f:(\U,\0)\rightarrow (\C, 0)$ is a complex analytic function which is not locally constant. In order to appreciate the inductive applications of the theorem, one should note that, if $s:=\dim_0\Sigma f\geq 1$, then, for generic $z_0$, $\dim_\0\Sigma(f_{|_{V(z_0)}})=s-1$.

\begin{thm}\label{thm:leattach}\textnormal{(L\^e, \cite{leattach})} Suppose that $\dim_0\Sigma f$ is arbitrary, Then, for a generic non-zero linear form $z_0$, $F_{f, \0}$ is obtained up to homotopy from $F_{f_{|_{V(z_0)}}, \0}$ by attaching $\left(\Gamma_{f, z_0}^1\cdot V(f)\right)_\0$ $n$-cells.
\end{thm}
\begin{proof} We once again assume that $z_0$ is a coordinate. The main technical issue, which we will not prove, is that one needs to know that one may use a disk times a ball, rather than a ball itself, when defining the Milnor fiber. More precisely, we shall assume that, up to homotopy, $F_{f,0}$ is given by
$$
F'_{f,0}:=\left(\D_\delta\times B^{2n}_\epsilon\right)\cap f^{-1}(a),
$$
where $0<|a|\ll\delta\ll\epsilon\ll 1$. Note that $F_{f_{|_{V(z_0)}}, \0}=V(z_0)\cap F'_{f,0}$.

The idea of the proof is simple: one considers $r:=|z_0|^2$ on $F'_{f,0}$. As in our previous proof, there is the problem that $r$ has a critical point at each point where $z_0=0$. But, again, as in our previous proof, $0$ is a regular value of $z_0$ restricted to $F'_{f,0}$. Hence, for $0<\eta\ll |a|$, 
$$r^{-1}[0,\eta]\cap F'_{f,0}\cong F_{f_{|_{V(z_0)}}, \0}\times \D_\eta.
$$

One also needs to prove a little lemma that, for $a$, $\delta$, and $\epsilon$ as we have chosen them, $z_0$ itself has no critical points on $\left(\D_\delta\times S^{2n-1}_\epsilon\right)\cap f^{-1}(a)$.

Now, one lets the value of $r$ grow from $\eta$ to $\delta$. Note that the critical points of $z_0$ restricted to $F'_{f,0}$ occur precisely at points in
$$
V\left(\frac{\partial f}{\partial z_1}, \dots, \frac{\partial f}{\partial z_n}\right)\cap V(f-a) \ = \ \Gamma_{f, z_0}^1\cap V(f-a),
$$
and, by the choice of generic $z_0$ all of these critical points will be complex non-degenerate. The result follows.
\end{proof}

\medskip

\begin{rem} Since attaching $n$-cells does not affect connectivity in dimensions $\leq n-2$, by inductively applying the above attaching theorem, one obtains that the Milnor fiber of a hypersurface in $\C^{n+1}$ with a critical locus of dimension $s$ is $(n-s-1)$-connected. Thus, one recovers the main result of \cite{katomatsu}.

It is also worth noting that L\^e's attaching theorem leads to a Lefschetz hyperplane result. It tells one that, for $k\leq n-2$, $H_k(F_{f,\0})\cong H_k(F_{f_{|_{V(z_0)}},\0})$ and that there is an exact sequence
$$
0\rightarrow  H_n(F_{f,\0})\rightarrow \Z^\tau \rightarrow H_{n-1}(F_{f_{|_{V(z_0)}},\0}) \rightarrow  H_{n-1}(F_{f,\0})\rightarrow 0,
$$
where $\tau=\left(\Gamma_{f, z_0}^1\cdot V(f)\right)_\0$.
\end{rem}

\medskip

\thmref{thm:leattach} seems to have been the first theorem about hypersurface singularities of arbitrary dimension that actually allowed for algebraic calculations. By induction, the theorem yields the Euler characteristic of the Milnor fiber and also puts bounds on the Betti numbers, such as $b_n(F_{f,\0})\leq\left(\Gamma_{f, z_0}^1\cdot V(f)\right)_\0$.

However, if $\dim_\0\Sigma f=0$, then $\left(\Gamma_{f, z_0}^1\cdot V(f)\right)_\0>\mu_f(\0)$ (this is {\bf not} obvious), and so the question is: are there numbers that we can calculate that are ``better'' than inductive versions of $\left(\Gamma_{f, z_0}^1\cdot V(f)\right)_\0$? We want numbers that are actual generalizations of the Milnor number of an isolated critical point.

\smallskip

Our answer to this is: yes -- the L\^e numbers, as we shall see in the next lecture.

\section{\bf Lecture 3: Proper intersection theory and L\^e numbers} 

Given a hypersurface $V(f)$ and a point $\mbf p\in V(f)$, if $\dim_{\mbf p}\Sigma f=0$, then the Milnor number of $f$ at $\mbf p$ provides a great deal of information about the local ambient topology of $V(f)$ at $\mbf p$.

But now, suppose that $s:=\dim_{\mbf p}\Sigma f>0$. What data should replace/generalize a number associated to a point? For instance, suppose $s=1$. A reasonable hope for ``good data'' to associate to $f$ at $\mbf p$ would be to assign a number to each irreducible component curve of $\Sigma f$ at $\mbf p$, and also assign a number to $\mbf p$. More generally, if $s$ is arbitrary, one could hope to produce effectively calculable topologically important data which consists of analytic sets of dimensions $0$ through $s$, with numbers assigned to each irreducible component.

\smallskip

This is what the {\bf L\^e cycles}, $\Lambda_{f, \mbf z}^s$, ..., $\Lambda_{f, \mbf z}^1$, $\Lambda_{f, \mbf z}^0$ (\cite{levar1}, \cite{levar2}, \cite{lecycles}), give you.

\smallskip

We briefly need to discuss what analytic cycles are, and give a few basic properties. Then we will define the L\^e cycles and the associated {\bf L\^e numbers}, and calculate some examples.

\smallskip

We need to emphasize that, in this lecture and the last one, when we write $V(\alpha)$, where $\alpha$ is an ideal (actually a coherent sheaf of ideals in $\mathcal O_\U$) we mean $V(\alpha)$ as a scheme, not merely an analytic set, i.e., we keep in mind what the defining ideal is.

\smallskip

\noindent{\rule{1in}{1pt}}

We restrict ourselves to the case of analytic cycles in an open subset $\U$ of some affine space $\C^N$. An {\it analytic cycle} in
$\U$ is a formal sum $\sum m_{{}_V} [V]$, where the $V$'s are (distinct) irreducible analytic
subsets of $\U$, the $m_{{}_V}$'s are integers, and the collection $\{V\}$ is a locally
finite collection of subsets of $\U$.  As a cycle is a locally finite sum, and as we
will normally be concentrating on the germ of an analytic space at a point, usually we
can safely assume that a cycle is actually a finite formal sum. If $C=\sum m_{{}_V} [V]$, we write $C_{{}_V}$ for the coefficient of $V$ in $C$, i.e., $C_{{}_V}=m_{{}_V}$.

For clarification of what structure we are considering, we shall at times enclose
cycles in square brackets, [ ] , and analytic sets in a pair of vertical lines, $| |$; with this notation, 
$$\left|\sum m_{{}_V} [V]\right| \ = \ \bigcup_{m_{{}_V}\neq 0}V.$$
Occasionally, when the notation becomes cumbersome, we shall simply state explicitly
whether we are considering  V  as a scheme, a cycle, or a set.

Essentially all of the cycles that we will use will be of the form $\sum m_{{}_V} [V]$, where all of the $V$'s have the same dimension $d$ (we say that the cycle is of pure dimension $d$) and $m_{{}_V}\geq 0$ for all $V$ (a non-negative cycle).

We need to consider not necessarily reduced complex analytic spaces (analytic schemes) $(X, \mathcal O_X)$ (in the sense of \cite{graurem1} and  \cite{graurem2}), where $X\subseteq\U$.  Given an analytic space, $(X, \mathcal O_{{}_X})$, we wish to define the cycle associated
to $(X, \mathcal O_{{}_X})$.  The cycle is defined to be the sum of the irreducible components, $V$, each one with a coefficient $m_{{}_V}$, which is its geometric multiplicity, i.e., $m_{{}_V}$ is how many times that component should be thought of as being there.

In the algebraic context, this is given by Fulton in section
1.5 of \cite{fulton} as  $$[X] \ := \ \sum m_{{}_V} [V] ,$$  where the $V$'s run over all
the irreducible components of $X$, and $m_{{}_V}$ equals the length of the ring $\mathcal
O_{{}_{V, X}}$, the Artinian local ring of $X$ along $V$.  In the analytic context, we wish to
use the same definition, but we must be more careful in defining the $m_{{}_V}$. Define
$m_{{}_V}$  as  follows.  Take a point $\mathbf p$ in $V$.  The germ of $V$ at $\mathbf p$
breaks up into irreducible germ components $(V_\mathbf p)_i$.  Take any one of the
$(V_\mathbf p)_i$ and let $m_{{}_V}$ equal the Artinian local length of the ring $(\mathcal O_{{}_{X, \mathbf
p}})_{(V_\mathbf p)_i}$ (that is, the local ring of $X$ at $\mathbf p$ localized at the
prime corresponding to $(V_\mathbf p)_i$).  This number is independent of the point
$\mathbf p$ in $V$ and the choice of $(V_\mathbf p)_i$.

\smallskip

Note that, in particular, if $\mbf p$ is an isolated point in $V(\alpha)$, then the coefficient of $\mbf p$ (really $\{\mbf p\}$) in $[V(\alpha)]$ is given by
$$[V(\alpha)]_\mbf p \ = \ \dim_{\C}\frac{\C\{z_0-p_0, \dots, z_n-p_n\}}{\alpha}.
$$

\smallskip 

Two cycles $C:=\sum m_{{}_V} [V]$ and $D:=\sum m_{{}_W} [W]$, of pure dimension $a$ and $b$, respectively, in $\U$ are said to {\bf intersect properly} if and only if, for all $V$ and $W$, $\dim (V\cap W)=a+b-N$ (recall that $N$ is the dimension of $\U$). 

When $C$ and $D$ intersect properly, there is a well-defined {\bf intersection product} which yields an {\bf intersection cycle} ({\bf not} a rational equivalence class); this intersection cycle is denoted by $(C\cdot D; \U)$ or simply $C\cdot D$ if the ambient complex manifold is clear. See Fulton \cite{fulton}, Section 8.2 and pages 207-208. Recalling our earlier notation, if $V$ is an irreducible component of $|C|\cap|D|$, then $(C\cdot D)_{{}_V}$ is the coefficient of $V$ in the intersection cycle. In particular, if $C$ and $D$ intersect properly in an isolated point $\mbf p$, then $(C\cdot D)_{\mathbf p}$ is called the {\bf intersection number of $C$ and $D$ at $\mbf p$.}

\medskip

We will now give some properties of intersection cycles and numbers. All of these can found in, or easily derived from, \cite{fulton}. We assume that all intersections written below are proper.

\begin{enumerate}

\item Suppose that $Y$ and $Z$ are irreducible analytic sets, and that $V$ is an irreducible component of their proper intersection. Then, $([Y]\cdot [Z])_V\geq 1$, with equality holding if and only if, along a generic subset of $V$, $Y$ and $Z$ are smooth and intersect transversely.

\medskip

\item If $f, g \in \mathcal O_U$, then $[V(fg)] =
[V(f)] + [V(g)]$; in particular, $[V(f^m)] = m[V(f)]$.

\medskip

\item $C\cdot D= D\cdot C$,  \ $(C\cdot D)\cdot E=C\cdot (D\cdot E)$,  \ and  
$$C\cdot\sum_i m_iD_i\ = \ \sum_i  m_i(C\cdot D_i).$$

\medskip

\item {\bf Locality}: Suppose that $Z$ is a component of $|C\cdot D|$ and that $\W$ is an open subset of $\U$ such that $Z\cap\W\neq\emptyset$ and $Z\cap\W$ is irreducible in $\W$. Then, 
$$
(C\cap\W \ \cdot \ D\cap\W; \ \W)_{Z\cap W} \ = \ (C\cdot \ D; \ \U)_{Z} 
$$

\medskip

\item If $f$ contains no isolated or embedded components of $V(\alpha)$, then $V(\alpha)\cdot V(f) = V(\alpha+\langle f\rangle)$. In particular, if $f_1, f_2, \dots f_k$ is a regular sequence, then
$$
V(f_1)\cdot V(f_2)\cdot \ldots \cdot V(f_k) \ = \ [V(f_1, f_2, \dots, f_k)].
$$

\medskip

\item {\bf Reduction to the normal slice}: Let $Z$ be a $d$-dimensional component of $C\cdot D$. Let $\mbf p$ be a smooth point of $Z$. Let $M$ be a {\bf normal slice} to $Z$ at $\mbf p$, i.e., let $M$ be a complex submanifold of $\U$ which transversely intersects $Z$ in the isolated point $\mbf p$. Furthermore, assume that $M$ transversely the smooth parts of $|C|$ and $|D|$ in an open neighborhood of $\mbf p$. Then, 
$$
(C\cdot \ D; \ \U)_{Z}  \ = \ \big((C\cdot M) \ \cdot  \ (D\cdot M); \ M)_{\mbf p}.
$$

\medskip

\item {\bf Conservation of number}: Let $E$ be a purely $k$-dimensional cycle in $\U$. Let $g_1(\mbf z, t)$, $g_2(\mbf z, t)$, \dots, $g_k(\mbf z, t)$ be in ${\mathcal O}_{\U\times{\D^{{}^\circ}}}$, for some open disk $\D^\circ$ containing the origin in $\C$. For fixed $t\in\D^\circ$, let $C_t$ be the cycle $[V(g_1(\mbf z, t), g_2(\mbf z, t),\dots, g_k(\mbf z, t))]$ in $\U$. Assume that $E$ and $C_0$ intersect properly in the isolated point $\mbf p$.

Then,
$$
(E\cdot C_0)_{\mbf p} \ = \ \sum_{\mbf q\in B_\epsilon^\circ(\mbf p)\cap|E|\cap |C_t|}\left(E\cdot C_t\right)_\mbf q,
$$
for $|t|\ll\epsilon\ll 1$.
\medskip

\item Suppose that $Z$ is a curve which is irreducible at $\mbf p$. Let $\mbf r(t)$ be a reduced parametrization of the germ of $Z$ at $\mbf p$ such that $\mbf r(0)=\mbf p$. (Here, by reduced, we mean that if $\mbf r(t)=\mbf p+\mbf a_1t+\mbf a_2t^2+\cdots$, then the exponents of the non-zero powers of $t$ with non-zero coefficients have no common factor, other than $1$.)  Suppose that $f\in\mathcal O_\U$ is such that that $V(f)$ intersects $Z$ in the isolated point $\mbf p$. Then,
$$
(Z\cdot V(f))_\mbf p \ = \ \operatorname{mult}_t f(\mbf r(t)),
$$
that is, the exponent of the lowest power of $t$ that appears in $f(\mbf r(t))$.

\bigskip

\begin{exer} Use the last property of intersection numbers above to show the following:

\begin{enumerate}

\item Suppose that $C=\sum_W m_W[W]$ is a purely $1$-dimensional cycle, and that $C$ properly intersects $V(f)$ and $V(g)$ at a point $\mbf p$. Suppose that, for all $W$, $(W\cdot V(f))_{\mbf p}<(W\cdot V(g))_{\mbf p}$. Then, $C$ properly intersects $V(f+g)$ at $\mbf p$ and $(C\cdot V(f+g))_{\mbf p}=(C\cdot V(f))_{\mbf p}$.

\medskip

\item Suppose that $C$ is a purely $1$-dimensional cycle, that $|C|\subseteq V(f)$, and that $C$ properly intersects $V(g)$ at a point $\mbf p$. Then, $C$ properly intersects $V(f+g)$ at $\mbf p$ and $(C\cdot V(f+g))_{\mbf p}= (C\cdot V(g))_{\mbf 0}$.
\end{enumerate}
\end{exer}

\end{enumerate}

\smallskip

\noindent{\rule{1in}{1pt}}

We are now (almost) ready to define the L\^e cycles and L\^e numbers. However, first, we need a piece of notation and we need to define the {\bf (relative) polar cycles}.

\medskip

Suppose once again that $\U$ is an open neighborhood of the origin in $\C^{n+1}$, and that \hbox{$f:(\U,\0)\rightarrow (\C,0)$} is a complex analytic function, which is not locally constant. We use coordinates $\mbf z=(z_0, \dots, z_n)$ on $\U$. We will at times assume, after possibly a linear change of coordinates, that $\mbf z$ is generic in some sense with respect to $f$ at $\0$. As before, we let $s:=\dim_\0\Sigma f$.

\smallskip

If $C=\sum_V m_V[V]$ is a cycle in $\U$ and $Z$ is an analytic subset of $\U$, then we let
$$
C_{{}_{\subseteq Z}} \ = \ \sum_{V\subseteq Z}m_V[V] \hskip 0.4in\textnormal{and}\hskip 0.4in 
C_{{}_{\not\subseteq Z}} \ = \ \sum_{V\not\subseteq Z}m_V[V].
$$

\smallskip

\begin{defn} For $0\leq k\leq n+1$, we define the {\bf $k$-th polar cycle of $f$ with respect to $\mbf z$} to be
$$
\Gamma^k_{f, \mbf z} \ := \ \left[V\left(\frac{\partial f}{\partial z_k}, \frac{\partial f}{\partial z_{k+1}}, \dots, \frac{\partial f}{\partial z_n}\right)\right]_{\not\subseteq \Sigma f}.
$$
Here, by $\Gamma^{n+1}_{f, \mbf z}$, we mean simply $[\U]$. Also, note that $\Gamma^0_{f, \mbf z}=0$.
\end{defn}

\medskip

As a set, this definition of $\Gamma^1_{f, \mbf z}$ agrees with our earlier definition of $\Gamma^1_{f, z_0}$; however, now we give this relative polar curve a cycle structure. If $z_0$ is generic enough, then all of the coefficients of components of the cycle $\Gamma^1_{f, \mbf z}$ will be $1$, but we typically do not want to assume this level of genericity.

Also note that every irreducible component of $V\left(\frac{\partial f}{\partial z_k}, \frac{\partial f}{\partial z_{k+1}}, \dots, \frac{\partial f}{\partial z_n}\right)$ necessarily has dimension at least $k$; hence, for $k\geq s+1$, there can be no components contained in $\Sigma f$ near the origin. Therefore, near the origin, for $k\geq s+1$,
$$
\Gamma^k_{f, \mbf z} \ := \ \left[V\left(\frac{\partial f}{\partial z_k}, \frac{\partial f}{\partial z_{k+1}}, \dots, \frac{\partial f}{\partial z_n}\right)\right].
$$

\bigskip

\begin{exer} In this exercise, you will be asked to prove what we generally refer to as the {\it Teissier trick}, since it was first proved by Teissier in \cite{teissiercargese} in the case of isolated critical points, but the proof is the same for arbitrary $s$.

Suppose that $\dim_\0\Gamma^1_{f, \mbf z}\cap V(f)\leq 0$. Then, $\dim_\0\Gamma^1_{f, \mbf z}\cap V(z_0)\leq 0$, $\dim_\0\Gamma^1_{f, \mbf z}\cap V\left(\frac{\partial f}{\partial z_0}\right)\leq 0$, and 
$$
\left(\Gamma^1_{f, \mbf z}\cdot V(f)\right)_\0 \ = \ \left(\Gamma^1_{f, \mbf z}\cdot V(z_0)\right)_\0  \ + \ \left(\Gamma^1_{f, \mbf z}\cdot V\left(\frac{\partial f}{\partial z_0}\right)\right)_\0 .
$$
(Hint: Parameterize the irreducible components of the polar curve and use the Chain Rule for differentiation.)

In particular, if $\dim_\0\Gamma^1_{f, \mbf z}\cap V(f)\leq 0$ and $\Gamma^1_{f, \mbf z}\neq 0$, then 
$$
\left(\Gamma^1_{f, \mbf z}\cdot V(f)\right)_\0 \ > \ \left(\Gamma^1_{f, \mbf z}\cdot V\left(\frac{\partial f}{\partial z_0}\right)\right)_\0 .
$$
\end{exer}

\bigskip

\begin{exer} Suppose that $k\leq n$, and that $\Gamma^{k+1}_{f, \mbf z}$ is purely $(k+1)$-dimensional and is intersected properly by $V\left(\frac{\partial f}{\partial z_k}\right)$.

Prove that
$$
\left(\Gamma^{k+1}_{f, \mbf z}\cdot V\left(\frac{\partial f}{\partial z_k}\right)\right)_{\not\subseteq \Sigma f} \ = \ \Gamma^{k}_{f, \mbf z}.
$$
\end{exer}

\bigskip

\begin{defn} Suppose that $k\leq n$, and that $\Gamma^{k+1}_{f, \mbf z}$ is purely $(k+1)$-dimensional and is intersected properly by $V\left(\frac{\partial f}{\partial z_k}\right)$. 

 Then we say that the  {\bf $k$-dimensional L\^e cycle exists}  and define it to be
$$
\Lambda^k_{f, \mbf z} \ := \ \left(\Gamma^{k+1}_{f, \mbf z}\cdot V\left(\frac{\partial f}{\partial z_k}\right)\right)_{\subseteq \Sigma f}.
$$

Hence,
$$
\Gamma^{k+1}_{f, \mbf z}\cdot V\left(\frac{\partial f}{\partial z_k}\right) \ = \ \Gamma^{k}_{f, \mbf z} \ + \ \Lambda^{k}_{f, \mbf z}.
$$
\end{defn}

\medskip

\begin{rem} Note that, if $\Lambda^s_{f, \mbf z}$ exists, then
$$
\Gamma^{s+1}_{f, \mbf z} \ = \ V\left(\frac{\partial f}{\partial z_{s+1}}, \dots, \frac{\partial f}{\partial z_{n}}\right)
$$
is purely $(s+1)$-dimensional and, for all $k\geq s$, $\Gamma^{k+1}_{f, \mbf z}$ is purely $(k+1)$-dimensional and is intersected properly by $V\left(\frac{\partial f}{\partial z_k}\right)$.

The point is that saying that $\Lambda^s_{f, \mbf z}$ exists implies that, for $s+1\leq k\leq n$,
 $\Lambda^k_{f, \mbf z}$ exists and is $0$.
 
 Furthermore, you should note that if  $\Lambda^k_{f, \mbf z}$ for all $k\leq s$, then each $\Lambda^k_{f, \mbf z}$ is purely $k$-dimensional and $\Sigma f=\bigcup_{k\leq s}\left|\Lambda^k_{f, \mbf z}\right|$.
 \end{rem}

\bigskip

\begin{exer} We wish to use intersection cycles to quickly show that the Milnor number is {\it upper-semicontinuous} in a family.

\begin{enumerate}

\item Suppose that $\dim_\0 \Sigma\big(f_{|_{V(z_0)}}\big)\leq 0$. Show that
$$
\mu_\0\big(f_{|_{V(z_0)}}\big) \ = \ \left(\Gamma^1_{f, \mbf z}\cdot V(z_0)\right)_\0 \ + \  \left(\Lambda^1_{f, \mbf z}\cdot V(z_0)\right)_\0.
$$

\medskip

\item
Suppose that we have a complex analytic function $F:(\D^\circ\times\U, \D^\circ\times\{\0\})\rightarrow(\C, 0)$. For each $t\in \D^\circ$, define $f_t:(\U,\0)\rightarrow(\C,0)$ by $f_t(\mbf z):=F(t,\mbf z)$, and assume that $\dim_\0\Sigma f_t=0$. Thus, $f_t$ defines a one-parameter family of isolated singularities.

Show, for all $t$ such that $|t|$ is sufficiently small, that  $\mu_\0(f_0) \geq \mu_\0(f_t)$, with equality if and only $\Gamma^1_{f, \mbf z}=0$ and $\D^\circ\times\U$ is the only component of $\Sigma F$ (near $\0$).

\end{enumerate}
\end{exer}

\medskip

\begin{defn} Suppose that, for all $k$ such that $0\leq k\leq n$, $\Lambda^k_{f, \mbf z}$ exists, 
 and 
 $$\dim_\0\Lambda^k_{f, \mbf z}\cap V(z_0, \dots, z_{k-1})\leq 0\hskip 0.2in \textnormal{and}\hskip 0.2in \dim_\0\Gamma^k_{f, \mbf z}\cap V(z_0, \dots, z_{k-1})\leq 0.
 $$
 
 Then, we say that the {\bf L\^e numbers} and {\bf polar numbers} of $f$ with respect to $\mbf z$ {\bf exist} in a neighborhood of $\0$ and define them, respectively, for $0\leq k\leq n$ and for each $\mbf p=(p_0, \dots, p_n)$ near the origin, to be
 $$
 \lambda^k_{f, \mbf z}(\mbf p):= \left(\Lambda^k_{f, \mbf z}\cdot V(z_0-p_0, \dots, z_{k-1}-p_{k-1})\right)_\mbf p
 $$
 and
  $$
 \gamma^k_{f, \mbf z}(\mbf p):= \left(\Gamma^k_{f, \mbf z}\cdot V(z_0-p_0, \dots, z_{k-1}-p_{k-1})\right)_\mbf p.
 $$
 \end{defn}

\medskip

\begin{rem} It is the L\^e numbers which will serve as our generalization of the Milnor of an isolated critical point. 

However, the existence of the polar numbers tells us that our coordinates are generic enough for many of our results to be true. In general, the condition that we will require of our coordinates -- which is satisfied generically -- will be that the L\^e and polar numbers, $\lambda^k_{f, \mbf z}(\0)$ and  $\gamma^k_{f, \mbf z}(\0)$, exist for $1\leq k\leq s$. (The existence when $k=0$ is automatic.) Note that when $s=0$, there is no requirement.
\end{rem}

\medskip

\begin{exer} Suppose that $s=1$. Show that the condition that $\lambda^1_{f, \mbf z}(\0)$ and  $\gamma^1_{f, \mbf z}(\0)$ exist is equivalent to requiring $\dim_\0\Sigma\big(f_{|_{V(z_0)}}\big)=0$.
\end{exer}

\medskip

Now we wish to look at three examples of L\^e cycle and L\^e number calculations.

\begin{exm}
Suppose that $s=0$. Then,
regardless of the coordinate system $\mathbf z$, the
only possibly non-zero L\^e number is $\lambda^0_{f, \mathbf z}(\mathbf 0)$.  Moreover, as
$V\left(\frac{\partial f}{\partial z_0} , \frac{\partial f}{\partial z_1}, \dots ,
\frac{\partial f}{\partial z_n}\right)$ is $0$-dimensional,   $\Gamma^1_{f, \mathbf z}=V\left(\frac{\partial
f}{\partial z_1}, \dots , \frac{\partial f}{\partial z_n}\right)$ is a $1$-dimensional complete intersection, and so has no components contained in $\Sigma f$ and has no embedded components. 

Therefore,   $$\lambda^0_{f, \mathbf z}(\mathbf
0) = \left(\Gamma^1_{f, \mathbf z} \cdot V\left(\frac{\partial f}{\partial
z_0}\right)\right)_\mathbf 0 =  \left(V\left(\frac{\partial f}{\partial z_1}, \dots ,
\frac{\partial f}{\partial z_n}\right) \cdot V\left(\frac{\partial f}{\partial
z_0}\right)\right)_\mathbf 0 =$$ $$\left(V\left(\frac{\partial f}{\partial z_0}, \dots ,
\frac{\partial f}{\partial z_n}\right)\right)_\mathbf 0 = \textnormal{ the Milnor number of  }f
\textnormal{ at } \mathbf 0 .$$
\end{exm}

\medskip

\begin{exm}\label{exm:degen} Let $f = y^2 - x^a - tx^b$, where $a > b > 1$.  We fix the
coordinate system $(t,x,y)$ and will suppress any further reference to it.

$$\Sigma f = V( -x^b, \ -ax^{a-1} - btx^{b-1}, \ 2y) = V(x,y) .$$  $$\Gamma^2_f =
V\left(\frac{\partial f}{\partial y}\right) = V(2y) = V(y) .$$ $$\Gamma^2_f \cdot
V\left(\frac{\partial f}{\partial x}\right) =  V(y) \cdot V(-ax^{a-1} - btx^{b-1}) =
V(y) \cdot (V(-ax^{a-b} -bt) + V(x^{b-1}))  =$$  $$V(-ax^{a-b} - bt, y) + (b-1)V(x, y)
= \Gamma^1_f + \Lambda^1_f .$$  $$\Gamma^1_f \cdot V\left(\frac{\partial f}{\partial
t}\right) = V(-ax^{a-b} - bt, y) \cdot V(-x^b) = bV(t,x,y) = b[\mathbf 0] = \Lambda^0_f
.$$  Thus, $\lambda^0_f(\mathbf 0) = b$ and $\lambda^1_f(\mathbf 0) = b-1$.  

\smallskip

Notice that
the exponent $a$ does not appear;  this is actually good, for  $f= y^2 - x^a - tx^b =
y^2 -x^b(x^{a-b} - t)$ which, after an analytic coordinate change at the origin, equals
$y^2 - x^bu$.
\end{exm}

\medskip

\begin{exm}\label{exm:fmcone} Let $f = y^2 - x^3 - (u^2 + v^2 + w^2)x^2$
and fix the coordinates $(u, v, w, x, y)$. $$\Sigma f = V(-2ux^2, \ -2vx^2, \ -2wx^2, \
-3x^2 - 2x(u^2 + v^2 + w^2), \ 2y)  = V(x,y) .$$

As $\Sigma f$ is three-dimensional, we begin our calculation with $\Gamma^4_f$.

$$\Gamma^4_f = V(-2y) = V(y) .$$

\vskip .1in

$$\Gamma^4_f \cdot V\left(\frac{\partial f}{\partial x}\right) = V(y) \cdot V( -3x^2 -
2x(u^2 + v^2 + w^2)) = $$ $$V(-3x - 2(u^2 + v^2 + w^2), y) + V(x,y) = \Gamma^3_f +
\Lambda^3_f .$$

\vskip .1in

$$\Gamma^3_f \cdot V\left(\frac{\partial f}{\partial w}\right) =  V(-3x - 2(u^2 + v^2 +
w^2), y) \cdot V(-2wx^2) = $$ $$V(-3x - 2(u^2 + v^2), w, y) + 2V(u^2 + v^2 + w^2, x, y)
= \Gamma^2_f + \Lambda^2_f .$$

\vskip .1in

$$\Gamma^2_f \cdot V\left(\frac{\partial f}{\partial v}\right) = V(-3x - 2(u^2 + v^2),
w, y) \cdot V(-2vx^2) = $$ $$V(-3x-2u^2, v, w, y) + 2V(u^2 + v^2, w, x, y) = \Gamma^1_f
+ \Lambda^1_f.$$

\vskip .1in

$$\Gamma^1_f \cdot V\left(\frac{\partial f}{\partial u}\right) = V(-3x-2u^2, v, w, y)
\cdot V(-2ux^2) = $$ $$V(u, v, w, x, y) + 2V(u^2, v, w, x, y) = 5[\mathbf 0] =
\Lambda^0_f.$$

Hence, $\Lambda^3_f = V(x,y)$, $\Lambda^2_f = 2V(u^2 + v^2 + w^2, x, y) =$ a cone (as a
set), $\Lambda^1_f = 2V(u^2 + v^2, w, x, y)$, and $\Lambda^0_f = 5[\mathbf 0]$.  Thus, at
the origin, $\lambda^3_f = 1$, $\lambda^2_f = 4$, $\lambda^1_f = 4$, and $\lambda^0_f =
5$.

Note that $\Lambda^1_f$ depends on the choice of coordinates - for, by symmetry, if we
re-ordered $u, v,$ and $w$, then $\Lambda^1_f$ would change correspondingly.  Moreover,
one can check that this is a generic problem.

Such ``non-fixed'' L\^e cycles arise from the absolute polar varieties of L\^e and Teissier (see \cite{leens}, \cite{teissiervp1}, \cite{teissiervp2}) of the higher-dimensional L\^e cycles.  For instance, in the
present case, $\Lambda^2_f$ is a cone, and its one-dimensional polar variety varies
with the choice of coordinates, but generically always consists of two lines; this is
the case for $\Lambda^1_f$ as well.  Though the L\^e cycles are not even generically
fixed, the L\^e numbers are, of course, generically independent of the coordinates.
\end{exm}

\smallskip

\noindent{\rule{1in}{1pt}}

\smallskip

Of course, you should be asking yourself: what does the calculation of the L\^e numbers tell us? We shall discuss this in the next lecture.

\section{\bf Lecture 4: Properties of L\^e numbers and vanishing cycles} 

Now that we know what L\^e cycles and L\^e numbers are, the question is: what good are they?

\medskip

Throughout this section, we will be in our usual set-up. We let $\U$ be an open neighborhood of the origin in $\C^{n+1}$, $f:(\U,\0)\rightarrow(\C,0)$ is a complex analytic function which is not locally constant, $s:=\dim_0\Sigma f$, and $\mbf z=(z_0, \dots, z_n)$ is a coordinate system on $\U$.

\smallskip

{\bf We shall also assume throughout this section, for all $k\leq s$, that $\lambda^k_{f,\mbf z}(\0)$ and $\gamma^k_{f,\mbf z}(\0)$ exist. We assume that $\U$ is chosen (or re-chosen) that $\dim\Sigma f=s$ (``globally'' in $\U$) and that $\lambda^k_{f,\mbf z}(\mbf p)$ and $\gamma^k_{f,\mbf z}(\mbf p)$ exist for all $k\leq s$ and $\mbf p\in\U$.
}

\smallskip 

All of the results on L\^e numbers given here can be found in \cite{lecycles}. However, we have recently replaced our previous assumptions with coordinates being {\it pre-polar} with the condition that L\^e numbers and polar numbers exist (as assumed above).

\medskip

Let us start with the generalization of Milnor's result that the Milnor fiber at an isolated critical point has the homotopy-type of a bouquet of spheres.

\smallskip

\begin{thm}\label{thm:massattach} The Milnor fiber $F_{f, \0}$ has the homotopy-type obtained by beginning with a point and successively attaching $\lambda^{s-k}_{f, \mbf z}(\0)$ $(n-s+k)$-cells for $0\leq k\leq s$.

In particular, there is a chain complex
$$
0\rightarrow\Z^{\lambda^{s}_{f, \mbf z}(\0)}\rightarrow\Z^{\lambda^{s-1}_{f, \mbf z}(\0)}\rightarrow\cdots\rightarrow\Z^{\lambda^{0}_{f, \mbf z}(\0)}\rightarrow 0
$$
whose cohomology at the ${\lambda^{k}_{f, \mbf z}(\0)}$ term is isomorphic to the reduced integral cohomology $\widetilde H^{n-k}(F_{f, \0})$. 

Thus, the Euler characteristic of the Milnor fiber is given by
$$
\chi(F_{f, \0}) \ = \ 1+ \sum_{0\leq k\leq s}(-1)^{n-k}{\lambda^{k}_{f, \mbf z}(\0)}.
$$
\end{thm}

\bigskip

\begin{exer} Let us go back and see what this tells us about our previous examples.

\begin{enumerate}
\item Look back at \exref{exm:degen}. We calculated that $\lambda^0_{f, \mbf z}(\0)=b$ and $\lambda^1_{f, \mbf z}(\0)=b-1$. Determine precisely the homology/cohomology of $F_{f,\0}$. Compare this with what \thmref{thm:massattach} tells us.

\medskip

\item Look back at \exref{exm:fmcone}. What is the Euler characteristic of the Milnor fiber at the origin? What upper-bounds do you obtain for the ranks of $H^1(F_{f, \0})$ and $H^4(F_{f, \0})$?
\end{enumerate}
\end{exer}

\medskip

In a recent paper with L\^e \cite{lemassey}, we showed the upper-bound of $\lambda^{s}_{f, \mbf z}(\0)$ on the rank of $\widetilde H^{n-s}(F_{f, \0})$ is obtained only in trivial cases, as given in the following theorem:

\smallskip

\begin{thm} Suppose that the rank of $\widetilde H^{n-s}(F_{f, \0})$ is $\lambda^{s}_{f, \mbf z}(\0)$. Then, near $\0$, the critical locus $\Sigma f$ is itself smooth, and $\Lambda^k_{f, \mbf z}=0$ for $0\leq k\leq  s-1$. 

This is equivalent saying that $f$ defines a family, parameterized by $\Sigma f$, of isolated hypersurface singularities with constant Milnor number. 
\end{thm}

\bigskip

Now we want to look at the extent to which the constancy of the L\^e numbers in a family controls the local, ambient topology in the family.

\smallskip

\begin{thm}\label{thm:leconst} Suppose that we have a complex analytic function $F:(\D^\circ\times\U, \D^\circ\times\{\0\})\rightarrow(\C, 0)$. For each $t\in \D^\circ$, define $f_t:(\U,\0)\rightarrow(\C,0)$ by $f_t(\mbf z):=F(t,\mbf z)$, and let $s=\dim_\0\Sigma f_0$. Suppose that the coordinates $\mbf z$ are chosen so that, for all $t\in\D^\circ$, for all $k$ such that $k\leq s$, $\lambda^k_{f_t,\mbf z}(\0)$ and $\gamma^k_{f_t,\mbf z}(\0)$ exist, and assume that, for each $k$, the value of $\lambda^k_{f_t,\mbf z}(\0)$ is constant as a function of $t$. 

Then,

\begin{enumerate}
\item The pair $(\D^\circ\times\U-\Sigma f,  \ \D^\circ\times\{\0\})$ satisfies Thom's $a_f$ condition at $\0$, i.e., all of the limiting tangent spaces from level sets of $f$ at the origin contain the $t$-axis (or, actually, its tangent line).
\item 

\begin{enumerate} 

\item The homology of the Milnor fibre of $f_t$ at the origin is
constant for all $t$ small.

\medskip

\item If $s \leq n-2$, then the fibre-homotopy-type of the Milnor fibrations of $f_t$ at
the origin is constant for all $t$ small;

\medskip

\item If $s \leq n-3$, then the diffeomorphism-type of the Milnor fibrations of $f_t$ at
the origin is constant for all $t$ small.

\end{enumerate}
\end{enumerate}
\end{thm}

\medskip

\begin{rem} Note that we do not conclude the constancy of the local, ambient topological-type. However, part of what the theorem says is that, if $s \leq n-3$, then at least the homeomorphism-type (in fact, diffeomorphism-type) of the small sphere minus the real link (i.e., the total space of the Milnor fibration) is constant.

It was an open question until 2004-2005 if the constancy of the L\^e numbers implies constancy of the topological-type. In \cite{bobadillale2}, Bobadilla proved this is the case when $s=1$, but, in \cite{bobadillale}, he produced a counterexample when $s=3$. On the other hand, Bobadilla did show that, for general $s$, in addition to the Milnor fibrations being constant, the homotopy-type of the real link in a family with constant L\^e numbers is also constant.
\end{rem}

\medskip

\thmref{thm:massattach} and \thmref{thm:leconst} are the reasons {\bf why} one wants algorithms for calculating L\^e numbers. 

\medskip

\begin{exer} The formulas in this exercise allow one to reduce some problems on a hypersurface to problems on a hypersurface with a singular set of smaller dimension. Suppose that $s\geq 1$.

\begin{enumerate}
\item Near $\0$, $\Sigma(f_{|_{V(z_0)}})=\Sigma f\cap V(z_0)$, and has dimension $s-1$. Let $\tilde{\mbf z}=(z_1, \dots, z_n)$ on $V(z_0)$. Then,
$$
\lambda^0_{f_{|_{V(z_0)}},\tilde{\mbf z}}(\0) \ = \ \gamma^1_{f, \mbf z}(\0) \ + \ \lambda^1_{f, \mbf z}(\0)
$$
and, for all $k\geq 1$,
$$
\lambda^k_{f_{|_{V(z_0)}},\tilde{\mbf z}}(\0) \ = \  \lambda^{k+1}_{f, \mbf z}(\0).
$$
\medskip
\item Suppose that $\gamma^1_{f, \mbf z}(\0)=0 $ or that $j>1+ \big(\lambda^0_{f, \mbf z}(\0)/\gamma^1_{f, \mbf z}(\0)\big)$. 

\smallskip

Then, in a neighborhood of the origin, 
$$\Sigma(f+z_0^j)=\Sigma f\cap V(z_0)\hskip 0.1in\textnormal{ and }\hskip 0.1in \dim_\0\Sigma(f+z_0^j)=s-1.$$
Furthermore, if we let $\tilde{\mbf z}$ denote the ``rotated'' coordinate system 
$$(z_1, z_2, \dots, z_n, z_0),$$ then
$$
\lambda^0_{f+z_0^j, \tilde{\mbf z}}(\0) \ = \ \lambda^0_{f, \mbf z}(\0)  \ + \ (j-1)\lambda^1_{f, \mbf z}(\0)
$$
and, for all $k\geq 1$,
$$
\lambda^k_{f+z_0^j, \tilde{\mbf z}}(\0) \ = \ (j-1)\lambda^{k+1}_{f, \mbf z}(\0).
$$
\end{enumerate}
\end{exer}

\smallskip

\noindent{\rule{1in}{1pt}}

We will now discuss the L\^e numbers from a very different, very sophisticated, point of view. We refer you once again to \cite{lecycles} if you wish to see many results for calculating L\^e numbers.

\bigskip

We want to discuss the relationship between the L\^e numbers and the vanishing cycles as a bounded, constructible complex of sheaves. As serious references to complex of sheaves, we recommend \cite{kashsch} and \cite{dimcasheaves}. Here, we will try to present enough material to give you the flavor of the machinery and results.

\bigskip

Suppose that $f$ has a non-isolated critical point at the origin. Then, at each point in $\Sigma f$, we have a Milnor fiber and a Milnor fibration. The question is: are there restrictions on the topology of the Milnor fiber at $\0$ imposed by the Milnor fibers at nearby points in $\Sigma f$?

This is a question of how local data patches together to give global data, which is exactly what sheaves encode. So it is easy to believe that sheaves would be useful in describing the situation. The stalk of a sheaf at a point describes the sheaf, in some sense, at the specified point. In our setting, at a point $\mbf p\in\Sigma f$, we would like for the stalk at $\mbf p$ to give us the cohomology of $F_{f, \mbf p}$. This means that, after we take a stalk, we need a chain complex. Thus, one is lead to consider complexes of sheaves, where the stalks of the individual sheaves are simply $\Z$-modules (we could use other base rings). Note that we are absolutely {\bf not} looking at coherent sheaves of modules over the structure sheaf of an analytic space. We should mention now that, for the most elegant presentation, it is essential that we allow our complex to have non-zero terms in {\bf negative} degrees.

As a quick example, on a complex space $X$, one important complex of sheaves is simply referred to as the {\it constant sheaf}, $\Z^\bullet_X$. This complex of sheaves is the constant sheaf $\Z_X$ in degree $0$ and $0$ in all other degrees. Frequently, we will {\it shift} this complex; the complex of sheaves $\Z^\bullet_X[k]$ is the complex which has $\Z_X$ in degree $-k$ (note the negation) and $0$ in all other degrees.

\smallskip

Suppose that we have a complex of sheaves (of $\Z$-modules) $\Adot$ on a space $X$. There are three cohomologies associated to $\Adot$:

\begin{enumerate}
\item The {\it sheaf cohomology}, that is, the cohomology of the complex: $\mbf H^k(\Adot)$.
\medskip

\item The {\it stalk cohomology} at each point $\mbf p\in X$: $H^k(\Adot)_{\mbf p}$. This is obtained by either taking stalks first and then cohomology of the resulting complex of $Z$-modules or by first taking the cohomology sheaves and then taking their stalks.
\medskip

\item The {\it hypercohomology} of $X$ with coefficients in $\Adot$: $\hyp^k(X; \Adot)$. This is a generalization of sheaf cohomology of the space. As an example, $\hyp^k(X; \Z^\bullet_X)$ simply yields the usual integral cohomology of $X$ in degree $k$.
\end{enumerate}

For all of these cohomologies, the convention on shifting tells us that the index is added to the shift to produce the degree of the unshifted cohomology, e.g., $\mbf H^k(\Adot[m])\cong \mbf H^{k+m}(\Adot)$.

It is standard that, for a subspace $Y\subseteq X$, one writes $\hyp^k(Y; \Adot)$ for the hypercohomology of $Y$ with coefficients in the restriction of $\Adot$ to $Y$; the point being that the restriction of the complex is usually not explicitly written in this context.

Our complexes of sheaves will be {\it constructible}. One thing that this implies is that restriction induces an isomorphism between the hypercohomology in a small ball and the stalk cohomology at the center of the ball, i.e., for all $\mbf p\in X$, there exists $\epsilon>0$ such that, for all $k$,
$$
\hyp^k(B^\circ_\epsilon(\mbf p); \Adot) \ \cong \ H^k(\Adot)_{\mbf p}.
$$

If $\W$ is an open subset of $X$ (and in other cases), the hypercohomology $\hyp^k(X, \W; \Adot)$ of the pair $(X, \W)$ is defined, and one has the expected long exact sequence
$$
\dots\rightarrow \hyp^{k-1}(\W; \Adot)\rightarrow\hyp^{k}(X,\W; \Adot)\rightarrow\hyp^{k}(X; \Adot)\rightarrow\hyp^{k}(\W; \Adot)\rightarrow\dots  \ .
$$

\smallskip

The {\it $k$-th support of $\Adot$}, $\operatorname{supp}^k\Adot$, is defined to be the closure of the set of points where the stalk cohomology in degree $k$ is not zero, i.e., 
$$
\operatorname{supp}^k\Adot \ := \ \overline{\{\mbf p\in X \ | \ H^k(\Adot)_{\mbf p}\neq 0\}}.
$$

The {\it support of $\Adot$}, $\operatorname{supp}\Adot$, is given by
$$
\operatorname{supp}\Adot \ := \bigcup_k \operatorname{supp}^k\Adot \ = \ \overline{\{\mbf p\in X \ | \ H^*(\Adot)_{\mbf p}\neq 0\}}.
$$
(We used above that $\Adot$ is bounded, so that the union is finite.)

The {\it $k$-cosupport of $\Adot$}, $\operatorname{cosupp}^k\Adot$, is defined to be
$$
\operatorname{cosupp}^k\Adot \ := \ \overline{\{\mbf p\in X \ | \ \hyp^k\big(B^\circ_\epsilon(\mbf p), B^\circ_\epsilon(\mbf p)-\{\mbf p\}; \ \Adot\big)\}}.
$$
\medskip

A {\bf perverse sheaf} (using middle perversity) $\Pdot$ is a bounded, constructible complex of sheaves which satisfies two conditions: the {\it support} and {\it cosupport} conditions, as given below. We remark that we set the dimension of the empty set to be $-\infty$, so that saying that a set has negative dimension means that the set is empty.

The support and cosupport conditions for a perverse sheaf are that, for all $k$,
\begin{enumerate} 
\item (support condition) $\dim\operatorname{supp}^k\Adot \ \leq -k$;

\medskip
\item (cosupport condition) $\dim\operatorname{cosupp}^k\Adot \ \leq k$.
\end{enumerate}
Note that the stalk cohomology is required to be $0$ in positive degrees. 

If $s:=\dim\supp\Pdot$, then the stalk cohomology is possibly non-zero only in degrees $k$ where $-s\leq k\leq 0$. In particular, if $\mbf p$ is an isolated point in the support of $\Pdot$, then the stalk cohomology of $\Pdot$ at $\mbf p$ is concentrated in degree $0$. Furthermore, the cosupport condition tells us that, for all $\mbf p\in X$,  for negative degrees $k$,
$$\hyp^k\big(B^\circ_\epsilon(\mbf p), B^\circ_\epsilon(\mbf p)-\{\mbf p\}; \ \Pdot\big)=0.$$

Note that the restriction of a perverse sheaf to its support remains perverse.

\medskip

\begin{exer} Suppose that $M$ is a complex manifold of pure dimension $d$. Show that the shifted constant sheaf $\Z_M^\bullet[d]$ is perverse.
\end{exer}

\medskip

What does any of this have to do with the Milnor fiber and L\^e numbers?

\medskip

Given an analytic $f:(\U, \0)\rightarrow (\C, 0)$, there are two functors, $\psi_f[-1]$ and $\phi_f[-1]$, from the category of perverse sheaves on $\U$ to the category of perverse sheaves on $V(f)$ called the (shifted) {\bf nearby cycles along $f$} and the {\bf vanishing cycles along $f$}, respectively. These functors encode, on a chain level, the \hbox{(hyper-})cohomology of the Milnor fiber and ``reduced'' (hyper-)cohomology of the Milnor fiber, and how they patch together. For all $\mbf p\in V(f)$, the stalk cohomology is what we want:
$$
H^k\left(\psi_f[-1]\Pdot\right)_\mbf p \ \cong \hyp^{k-1}(F_{f, \mbf p};\Pdot)
$$
and
$$
H^k\left(\phi_f[-1]\Pdot\right)_\mbf p \ \cong \hyp^{k}(B^\circ_\epsilon(\mbf p), F_{f, \mbf p};\Pdot).
$$

\medskip

\begin{exer} Recall from the previous exercise that $\Z_\U^\bullet[n+1]$ is perverse. As $\phi_f[-1]$ takes perverse sheaves to perverse sheaves, it follows that the complex of sheaves $\phi_f[-1]\Z_\U^\bullet[n+1]$ is a perverse sheaf on $V(f)$.
\begin{enumerate}
\item Show that, for all $\mbf p\in V(f)$, for all $k$, 
$$H^k(\phi_f[-1]\Z_\U^\bullet[n+1])_{\mbf p}\cong \widetilde H^{n+k}(F_{f, \mbf p}; \Z),$$
 and so $\supp\phi_f[-1]\Z_\U^\bullet[n+1] =V(f)\cap \Sigma f$.

\medskip

\item Explain why this tells you that $\widetilde H^j(F_{f, \0};\Z)$ is possibly non-zero only for $n-s\leq j\leq n$.
\end{enumerate}
\end{exer}

\medskip

\begin{exm} In this example, we wish to show some of the power of perverse techniques.

\smallskip

Suppose that $s=1$, so that $\Sigma f$ is a curve at the origin. Let $C_i$ denote the irreducible components of $\Sigma f$ at $\0$. Then, $\Lambda^1_{f, \mbf z}=\sum_i \mu_i^\circ[C_i]$, where $\mu_i^\circ$ is the Milnor number of $f$ restricted to a transverse hyperplane slice to $C_i$ at any point on $C_i-\{\0\}$ near $\0$. Furthermore, in a neighborhood of $\0$, $C_i-\{\0\}$ is a punctured disk, and there is an internal monodromy action $h_i:\Z^{\mu_i^\circ}\rightarrow \Z^{\mu_i^\circ}$ induced by following the Milnor fiber once around this puncture.

Let $\Pdot:=\left(\phi_f[-1]\Z_\U^\bullet[n+1]\right)_{|_{\Sigma f}}$. This is a perverse sheaf on $\Sigma f$. We are going to look at the long exact hypercohomology sequence of the pair 
$$\big(B_\epsilon\cap \Sigma f, \ (B_\epsilon-\{\0\})\cap\Sigma f\big),
$$
for small $\epsilon$. For convenience of notation, we will assume that we are working in a small enough ball around the origin, and replace $B_\epsilon\cap \Sigma f$ with simply $\Sigma f$, and $(B_\epsilon-\{\0\})\cap\Sigma f$ with $\Sigma f-\{\0\}=\bigcup_i (C_i-\{\0\})$.

Because we are working in a small ball, 
$$\hyp^k(\Sigma f; \Pdot) \ \cong \  H^k(\Pdot)_\0 \ \cong \ \widetilde H^{n+k}(F_{f, \0}; \Z).
$$

Furthermore, $\Pdot$ restricted to $C_i-\{\0\}$ is a local system (a locally constant sheaf), shifted into degree $-1$ with stalk cohomology $\Z^{\mu_i^\circ}$. Hence, $$\hyp^{-1}(C_i-\{\0\}; \Pdot)\cong \operatorname{ker}\{\operatorname{id}-h_i\},$$ and
$$\hyp^{-1}\big(\bigcup_i (C_i-\{\0\}); \Pdot\big) \ \cong \  \bigoplus_i \hyp^{-1}((C_i-\{\0\}); \Pdot) \ \cong \ \bigoplus_i\operatorname{ker}\{\operatorname{id}-h_i\}.
$$
Finally note that the cosupport condition tells us that $\hyp^{-1}(\Sigma f, \Sigma f-\{\0\};\Pdot)=0$. 

Thus, the portion of the long exact sequence on hypercohomology 
$$
 \hyp^{-1}(\Sigma f, \Sigma f-\{\0\};\Pdot)\rightarrow \hyp^{-1}(\Sigma f;\Pdot)\rightarrow\hyp^{-1}(\Sigma f-\{\0\};\Pdot)\rightarrow
$$
becomes
$$
0\rightarrow \widetilde H^{n-1}(F_{f, \0}; \Z)\rightarrow \bigoplus_i\operatorname{ker}\{\operatorname{id}-h_i\}\rightarrow.
$$
Even without knowing the $h_i$'s, this tells us that the rank of $\widetilde H^{n-1}(F_{f, \0}; \Z)$ is at most $\sum_i\mu_i^\circ$.

Note that, if all of the components of $\Sigma f$ are smooth, then $\lambda^1_{f, \mbf z}(\0)=\sum_i\mu_i^\circ$, and this bound agrees with what we obtain from \thmref{thm:massattach}, but -- if one of the $C_i$'s is singular at $\0$ -- then this example produces a better bound.

\smallskip

This result is quite complicated to prove without using perverse sheaves; see \cite{siersmavarlad}.
\end{exm}

\smallskip

\noindent{\rule{1in}{1pt}}

We have yet to tell you what the vanishing cycles have to do with the L\^e numbers.

\smallskip

Suppose that $X$ is a complex analytic subspace of $\U$, that $\Pdot$ is a perverse sheaf on $X$, and that $\mbf p\in X$. For a generic affine linear form $L$ such that $L(\mbf p)=0$, the point $\mbf p$ is an isolated point in the support of $\phi_L[-1]\Pdot$; for instance, $L$ can be chosen so that $V(L)$ is transverse to all of the strata -- near $\mbf p$ but not at $\mbf p$ -- of a Whitney stratification with respect to which $\Pdot$ is constructible. For such an $L$, since $\mbf p$ is an isolated point in the support of the perverse sheaf $\phi_L[-1]\Pdot$, the stalk cohomology of $\phi_L[-1]\Pdot$ at $\mbf p$ is concentrated in degree $0$.

Our coordinates $(z_0, z_1, \dots, z_n)$ have been chosen so that all of the iterated vanishing and nearby cycles below have $\0$ as an isolated point in their support.

As we showed in \cite{numinvar}, the connection with the L\^e numbers is:
$$
\lambda^0_{f, \mbf z}(\0) \ = \ \rank H^0\big(\phi_{z_0}[-1]\phi_f[-1]\Z_\U^\bullet[n+1]\big),
$$
$$
\lambda^1_{f, \mbf z}(\0) \ = \ \rank H^0\big(\phi_{z_1}[-1]\psi_{z_0}[-1]\phi_f[-1]\Z_\U^\bullet[n+1]\big),
$$
$$\vdots$$
$$
\lambda^s_{f, \mbf z}(\0) \ = \ \rank H^0\big(\phi_{z_s}[-1]\psi_{z_{s-1}}[-1]\dots\psi_{z_{0}}[-1]\phi_f[-1]\Z_\U^\bullet[n+1]\big).
$$
The chain complex in \thmref{thm:massattach} can be derived as a formal consequence of the equalities above.

\section{\bf Appendix} 

\noindent{\bf Curve Selection Lemma}:

\smallskip

The following lemma is an extremely useful tool; see \cite{milnorsing}, \S3 and \cite{looibook}, \S2.1. The complex analytic statement uses Lemma 3.3 of \cite{milnorsing}. Below, ${\stackrel{\circ}{\D}}_{{}_\delta}$ denotes an open disk of radius $\delta>0$ centered at the origin in $\C$.

\begin{lem}\label{lem:curveselection} {\rm (Curve Selection Lemma)} Let $\mbf p$ be a point in a real analytic manifold $M$. Let $Z$ be a semianalytic subset of $M$ such that $\mbf p\in\overline{Z}$. Then, there exists a real analytic curve $\gamma: [0, \delta)\rightarrow M$ with $\gamma(0)=\mbf p$ and $\gamma(t)\in Z$ for $t\in (0, \delta)$.

Let $\mbf p$ be a point in a complex analytic manifold $M$. Let $Z$ be a constructible subset of $M$ such that $\mbf p\in\overline{Z}$. Then, there exists a complex analytic curve $\gamma: {\stackrel{\circ}{\D}}_{{}_\delta}\rightarrow M$ with $\gamma(0)=\mbf p$ and $\gamma(t)\in Z$ for $t\in {\stackrel{\circ}{\D}}_{{}_\delta}-\{\0\}$.
\end{lem}

\smallskip

\noindent{\rule{1in}{1pt}}

\noindent{\bf Cones and joins}:

\smallskip

Recall that the abstract cone on a topological space $Y$ is the identification space
$$c(Y):=\frac{Y\times[0,1]}{Y\times\{1\}},$$
where the point (the equivalence class of) $Y\times\{1\}$ is referred to as the {\it cone point}. Then define the open cone to be $c^\circ(Y):=c(Y)\backslash\big(Y\times \{0\}\big)$.

If $Z\subseteq Y$, then $c(Z)\subseteq c(Y)$, where we consider the two cone points to be the same, and denote this common cone point simply by $*$. 

We define the cones on the pair $(Y, Z)$ to be triples, which include the cone point:
$$
c(Y,Z) \ := \ (c(Y),\, c(Z),\, *)\hskip 0.2in \textnormal{and}\hskip 0.2in c^\circ(Y,Z) \ := \ (c^\circ(Y),\, c^\circ(Z),\, *).
$$

\medskip

The {\it join} $X*Y$ of topological spaces $X$ and $Y$ is the space $X\times Y\times [0,1]$, where at one end the interval ``$X$ is crushed to a point'' and at the other end of the interval ``$Y$ is crushed to a point''. This means that we take the identification space obtained from $X\times Y\times [0,1]$ by identifying $(x_1, y, 0)\sim (x_2, y, 0)$ for all $x_1, x_2\in X$ and $y\in Y$, and also identifying $(x, y_1, 1)\sim (x, y_2, 1)$ for all $x\in x$ and $y_1, y_2\in Y$.

The join of a point with a space $X$ is just the cone on $X$. The join of the $0$-sphere (two discrete points) with $X$ is the suspension of $X$.

\smallskip

\noindent{\rule{1in}{1pt}}

\noindent{\bf Locally trivial fibrations}:

\smallskip

Suppose that $X$ and $Y$ are smooth manifolds, where $X$ may have boundary, but $Y$ does not. A smooth map $g:X\rightarrow Y$ is a {\it smooth trivial fibration} if and only if there exists a smooth manifold $F$, possibly with boundary, and a diffeomorphism $\alpha:F\times Y\rightarrow X$  such that the following diagram commutes:

$$
\begin{tikzcd}
F\times Y \arrow{r}{\alpha}[swap]{\cong} \arrow[swap]{dr}{\pi}  &X \arrow{d}{g} \\
& Y,
\end{tikzcd}
$$
where $\pi$ denotes projection.

A smooth map $g:X\rightarrow Y$ is a {\it smooth locally trivial fibration} if and only if, for all $y\in Y$, there exists an open neighborhood $B$ of $y$ in $Y$ such that the restriction $g_{|_{g^{-1}(B)}}: g^{-1}(B)\rightarrow B$ is a trivial fibration.

If the base space $Y$ is connected and $g:X\rightarrow Y$ is a smooth locally trivial fibration, it is an easy exercise to show that the diffeomorphism-type of the fibers $g^{-1}(y)$ is independent of $y\in Y$; in this case, any fiber is referred to as simply {\it the fiber of the fibration}.

\medskip

It is a theorem that a locally trivial fibration over a contractible base space is, in fact, a trivial fibration.

\medskip

Now we state the theorem of Ehresmann \cite{ehresmann}, which yields a common for method for concluding that a map is a locally trivial fibration.

\begin{thm}\label{thm:ehresmann} Suppose that $N$ is a smooth manifold, possibly with boundary,  $P$ is a smooth manifold, and $f: N\rightarrow P$  is a proper submersion. Then, $f$ it is a smooth, locally trivial fibration.
\end{thm}

Locally trivial fibrations over a circle are particularly easy to characterize. Begin with a smooth fiber $F$ and a diffeomorphism $\tau:F\rightarrow F$, called a {\it characteristic diffeomorphism}. Then there is a smooth locally trivial fibration 
$$
p: \frac{F\times[0,1]}{(x, 0)\sim (\tau(x),1)}\rightarrow S^1\subseteq\C
$$
given by $p([x,t]):= e^{2\pi it}$. A characteristic diffeomorphism describes how the fiber is ``glued'' to itself as one travels counterclockwise once around the base circle.

Every smooth locally trivial fibration $g$ over a circle is diffeomorphic to one obtained as above, but the  characteristic diffeomorphism $\tau$ is {\bf not} uniquely determined by $g$; however, the maps on the homology and cohomology of $F$ induced by $\tau$ {\bf are} independent of the choice of $\tau$. These induced maps are called the {\it monodromy automorphisms} of $g$. Thus, for each degree $i$, there are  well-defined monodromy maps $T_i: H_i(F;\Z)\rightarrow H_i(F;\Z)$ and $T^i: H^i(F;\Z)\rightarrow H^i(F;\Z)$.

\newpage
\bibliographystyle{plain}
\bibliography{Masseybib}
\end{document}